\documentclass[final,times]{elsarticle}
 
 \usepackage[dvips]{epsfig}
\usepackage{amscd}
\usepackage{amssymb}
\usepackage{amsthm}
\usepackage{amsmath}
\usepackage{latexsym}

\usepackage{upref}
\usepackage{hyperref}
\usepackage{color}

\theoremstyle{plain}
\newtheorem{thm}{Theorem}[section]
\theoremstyle{plain}
\newtheorem{lem}[thm]{Lemma}
\newtheorem{prop}[thm]{Proposition}
\newtheorem{cor}[thm]{Corollary}

\theoremstyle{definition}
\newtheorem{defi}{Definition}[section]
\newtheorem*{rem}{Remark}

\newenvironment{Assumptions}
{%
\setcounter{enumi}{0}

\begin{enumerate}}%
{\end{enumerate} }
{%
\setcounter{enumi}{0}

\begin{enumerate}}%
{\end{enumerate} }

\newcommand{\eps}{\ensuremath{\epsilon}}
\newcommand{\R}{\ensuremath{\mathbb{R}}}

\numberwithin{equation}{section} \allowdisplaybreaks

\journal{Journal of Differential Equations}
\begin{document}

\begin{frontmatter}

\title{Regularization by $\frac{1}{2}$-Laplacian and vanishing viscosity approximation of  HJB equations}
\author{Imran H. Biswas}

\address{Centre for Applicable Mathematics,
 Tata Institute of Fundamental Research,
  P.O.\ Box 6503, GKVK Post Office,
  Bangalore 560065, India}

\begin{abstract}
We investigate the regularizing effect of adding small fractional Laplacian, with critical fractional exponent $\frac 12$, to a general first order HJB equation. Our results include some regularity estimates for the viscosity solutions of such  perturbations, making the solutions classically well-defined. Most importantly, we use these regularity estimates to study the vanishing viscosity approximation to first order HJB equations by  $\frac 12$-Laplacian and derive an explicit  rate convergence for the vanishing viscosity limit.
\end{abstract}

\begin{keyword}
viscosity solutions,  HJB equations, vanishing viscosity,
 integro-partial differential equation, fractional Laplacian.

\MSC[2000]{45K05, 46S50, 49L20, 49L25, 91A23, 93E20}

\end{keyword}

\end{frontmatter}
\section{Introduction}
Within the field of fully nonlinear partial differential equations, the Hamilton-Jacobi-Bellman type equations are one of the most widely studied class. Among others, the notion of viscosity solutions has been of immense help to achieve a deeper understanding of fully nonlinear PDEs.  It is well documented in the literature that  the regularizing effect of adding small diffusion to first order fully nonlinear HJB equations has played a pivotal role in the development and understanding of viscosity solution theory. In this article we also set out to study a similar problem by adding a small fractional diffusion to a class of fully nonlinear first order HJB equations and investigate the regularizing effect and convergence properties of such approximations.   We are interested in the following initial value problem
      \begin{align}
  \label{eq:HJB-eq} \begin{cases}u_t + H(t,x, u(t,x), \nabla u(t,x)) &= 0\quad\text{if}\quad (t,x) \in (0,T]\times\mathbb{R}^n,\\
   u(0,x) &= u_0(x)
   \end{cases}
 \end{align} and it's vanishing viscosity approximation

 \begin{align}
 \label{eq:HJB-eq-viscous}\begin{cases}u_t^\epsilon + H(t,x, u^\epsilon(t,x), \nabla u^\epsilon(t,x))+\epsilon (-\Delta )^{\frac{s}{2}}u^\epsilon &= 0\quad\text{if}\quad (t,x) \in (0,T]\times\mathbb{R}^n,\\
 u(0,x) & = u_0(x).
 \end{cases}
 \end{align}

     In the above $T$ is a positive constant, the Hamiltonian $H$ is a real valued function on $\R\times\R^n\times\R\times \R^n$ and $\eps > 0$ is a small positive number. The precise structural assumptions on $H$ will be detailed in Section 2, but roughly speaking, it is a Lipschitz continuous function in all its variables and enjoys some monotonicity property in $u$.  The initial data $u_0(x)$ is a Lipschtiz continuous function on $\R^n$. The number $\frac s2$ in \eqref{eq:HJB-eq-viscous} is the fractional power of the diffusion operator and $s$ is supposed to be ranging within $[1,2]$.
     
      Among others, a rich source for equations of type \eqref{eq:HJB-eq} is the area of optimal control. The value function of a controlled dynamical system or that of a differential game solves an equation of the form \eqref{eq:HJB-eq}. Also, the equations of type \eqref{eq:HJB-eq-viscous} are of paramount importance due to their appearance in optimal control of stochastic dynamical systems with $\alpha$-stable noise.  The problem \eqref{eq:HJB-eq-viscous} is clearly a perturbation of  \eqref{eq:HJB-eq}. From the optimal control viewpoint, if the controlled deterministic dynamical system is perturbed by a small additive L\'{e}vy noise then the resulting value function of the perturbed control problem would satisfy an equation of type \eqref{eq:HJB-eq-viscous}.  Our aim in this article is to study the stability of such perturbation and it's regularizing effect on the value function.

      For $s=2$, the problem \eqref{eq:HJB-eq-viscous} becomes the classical parabolic approximation of  \eqref{eq:HJB-eq} and the classical theory for semilinear parabolic equations applies. As a result, for s=2, the Cauchy problem \eqref{eq:HJB-eq-viscous} is well-posed and the solution $u^\epsilon$ is smooth (cf. \cite{F'man:1973jk}). It is also well-known that the sequence of functions $(u^\epsilon)_{\epsilon> 0}$ converges locally uniformly to a function $u$ as $\eps\downarrow 0$, which is characterized  as the unique {\it viscosity solution} of \eqref{eq:HJB-eq}. There are a number of methods available (cf. \cite{evans:2010hj,Jakobsen:2005jy}) to estimate the rate of convergence, and one can optimally estimate the error to be of the order $\eps^{\frac 12}$.
      
      The case $s < 2$ is much less classical.  The operator $(-\Delta)^{s/2}$ has the following representation (cf. \cite{Landkof:1966}):
    \begin{align}
      \label{eq:non-local-representaion}    (-\Delta)^{\frac{s}{2}}u(x) = C(n,s)\int_{\mathbb{R}^n} \frac{u(x)-u(x+y)}{|y|^{n+s}}dy.
    \end{align} The constant $C(n,s)$ depends only on $n$ and $s$. The above integral in   \eqref{eq:non-local-representaion} should be understood in the principal value sense. Clearly, in view of \eqref{eq:non-local-representaion},  the problem \eqref{eq:HJB-eq-viscous} is non-local in nature or, in other words, an integro-partial differential equation. However, the notion of viscosity solution does make sense for such equations and the literature addressing this notion and related issues is fairly well developed by now. We refer to the articles \cite{BCI:P07, BI:P07, Imbert: 2005ft, Sayah:2010gt, Jakobsen:2005jy, silvestre:2010} and the references therein for more on this topic. The issues addressed in these papers range from standard wellposedness theory to more subtle questions related to regularity.
    
      For $1<s< 2$, the question on regularization was first answered by C. Imbert \cite{Imbert: 2005ft}. It was shown, under certain conditions, that the unique viscosity solution of  \eqref{eq:HJB-eq-viscous} is indeed of class $C^{1,2}$. In other words, the perturbed equation  \eqref{eq:HJB-eq-viscous} is classically welldefined and the perturbation has the same effect as classical parabolic regularization.  In  \cite{Imbert: 2005ft}, the author also gives a condition on the Hamiltonian $H$ under which $u^\epsilon$ becomes $C^\infty$. The $L^\infty$-error bound on $u^\epsilon- u$ is estimated to be of the order $\eps^{\frac{1}{s}}$. The error estimate for vanishing viscosity approximation in  \cite{Imbert: 2005ft} is optimal, and it was a significant improvement over the earlier result in \cite{Jakobsen:2005jy} which was of the order $\eps^{\frac{1}{2}}$. The results in \cite{Imbert: 2005ft} are new but mostly along the expected lines for the following reasons. The equation \eqref{eq:HJB-eq-viscous} could be seen as a perturbation of the fractional heat equation
       $$u_t +\eps (-\Delta )^{\frac{s}{2}}u = 0,$$
       
 and, if $ s> 1$, this equation has similar regularity property. Therefore, going by classical parabolic regularization results, it is only natural to anticipate  that  \eqref{eq:HJB-eq-viscous} will have smooth solutions and results in  \cite{Imbert: 2005ft} confirm this.  For $s < 1$, the dominant derivative is of first order and featured by the Hamiltonian $H$ and therefore it is not fair to expect any further regularization.  In this case,  it is well documented in the literature ( cf. \cite{Kiselev:2008}) that equation \eqref{eq:HJB-eq-viscous}  will not have smooth solutions in general, the viscosity solutions are at best Lipschtiz continuous for  Lipschtiz initial data.

       The case $s=1$, as has been rightly termed, is critical. The orders of the original HJB operator and added (nonlocal) fractional Laplace operator in  \eqref{eq:HJB-eq-viscous} are same, and it is a priori not clear at all whether there is any smoothing effect. On the intuitive level, one is more likely to think the opposite that there may not be any regularizing effect after all.  The problem of determining the regularity for this critical case is much more delicate and the strategy of \cite{Imbert: 2005ft} does not apply in this case. It is only recently that there has been a breakthrough by L.~ Silvestre\cite{silvestre:2010} on this question. In \cite{silvestre:2010}, the author shows if  $H$ is independent of $(t,x,u)$ then the unique viscosity solution of \eqref{eq:HJB-eq-viscous} is indeed $C^{1,\alpha}$. In other words, the viscosity solutions are regular enough to satisfy the equation in the classical sense.  We must admit that the techniques used by L. Silvestre are fairly delicate in nature.  The regularity estimate is obtained by establishing a diminish of oscillation lemma for the linearized version of \eqref{eq:HJB-eq-viscous}.

        In this article we will concentrate on the case of critical fractional order i.e $s= 1$ and extend the results of \cite{Imbert: 2005ft,silvestre:2010}. In other words, we want to investigate the regularity of the following problem:
                       \begin{align}
 \label{eq:HJB-eq-viscous-1}\begin{cases}u_t^\eps + H(t,x, u^\eps(t,x), \nabla u^\eps(t,x))+\eps (-\Delta )^{\frac{1}{2}}u^\eps(t,x) &= 0\quad\text{if}\quad (t,x) \in (0,T]\times\mathbb{R}^n,\\
 u^\eps(0,x) & = u_0(x).
 \end{cases}
 \end{align}
                     The contributions in this paper has two components. In the first part  we extend and adapt the methodology of \cite{silvestre:2010} to cases where the Hamiltonian $H$ can have dependence on $(t,x,u)$ as well as $\nabla u$ and
prove a $C^{1,\alpha}$ estimate for the viscosity solution. Secondly, we estimate the error $||u^\eps -u||_{L^\infty}$ for the vanishing viscosity approximation, which comes out be of type $C \epsilon |\log \eps|$.\footnote{It was brought to our notice by the referee that  the same results on error estimate have been derived earlier by Droniou and Imbert \cite{Dro2006}. }. This result on error estimate indeed establishes that, for critical exponent $s= 1$, the error estimate for vanishing viscosity approximation is not linear in $\epsilon$.
     
      \section{Technical framework and main results}

        We begin by introducing the notations that are going to be used in the rest of this paper. By $C, K, N$  we mean various constants depending on the data. There will be occasions where the constant may change from line to line but the notation is kept unchanged. The Euclidean norm on any $\mathbb{R}^d$-type space is
denoted by $|\cdot|$.  For any $r > 0$ and $x\in \R^n$, we use the notation $B_r(x)$ for the open ball of radius $r$ around $x$. In the case when $x= 0$, we simply write $B_r$ in place of $B_r(0)$ and define $Q_r =[-r, 0]\times B_r$. For any subset $Q\subset \mathbb{R}\times
\mathbb{R}^n$ and for any bounded, possibly vector valued
function $w$ on $Q$, we define the following norms:
\begin{align*}
	|w|_0 := \sup_{(t,x)\in Q} |w(t,x)|,\qquad
	|w|_{0,\alpha} = |w|_0 + \sup_{(t,x)\neq(s,y)}
	\frac{|w(t,x)-w(s,y)|}{|t-s|^{\alpha}+|x-y|^\alpha},
\end{align*} where $\alpha \in (0,1]$ is a constant. The function space $C^{1,\alpha}(Q)$ is the space of bounded and differentiable functions $w$  such that $D_{t,x} w = (\partial_t w, \nabla_x w)$ is  H\"{o}lder continuous of exponent $\alpha$. This space is endowed with the norm
    \begin{align*}
     ||w||_{C^{1,\alpha}(Q)} = |w|_0 + |D_{t,x}w|_{0,\alpha}.
    \end{align*} Denote by $C^{0,\alpha}(Q)$ the space of all functions on $Q$ such that $|w|_{0,\alpha} < \infty$.  Also denote the set of all upper and lower semicontinuous functions on $Q$ respectively by $USC(Q)$ and $LSC(Q)$. A lower index would mean polynomial growth at infinity, therefore the spaces $USC_p(Q)$ and $ LSC_p(Q)$ contain the functions $w$ respectively from $USC(Q)$ and $LSC(Q)$ satisfying the growth condition
    $$|w(t,x)| \le C (1+ |x|^p).$$
    We identify the spaces $USC_0(Q)$ and $LSC_0(Q)$  respectively with $USC_b(Q)$ and $LSC_b(Q)$; `$b$' is an index signifying boundedness.  We want the initial value problem  \eqref{eq:HJB-eq}, interpreted in the viscosity sense, to be well-posed and to have  Lipschitz continuous solutions. To this end, we list the following assumptions:
    \vspace{.5 cm}
    \begin{Assumptions}
    \item\label{A1} The Hamiltonian $H: \mathbb{R}\times\mathbb{R}^n\times\mathbb{R}\times \mathbb{R}^n\rightarrow \R$ is continuous and there is a positive constant $K $ such that
      $$ \sup_{t\in [0,T], x\in \R^n}|H(t,x,0, 0)| < K.$$
    \item \label{A2} There exists $\lambda\ge 0$ such that for all $(t,x,p)\in \R\times \R^n\times\R^n$ and $u,v \in \R$
         $$H(t,x,v,p)-H(t,x,u,p) = \lambda (v-u).$$
       \item\label{A3}There exists a constant $C > 0$ such that for all $(x,p,q)\in\R^n\times \R^n\times \R^n$ and $t,s \in \R$,
           $$|H(t,x,u,p)-H(s,y,u,p) |\le C(1+|p|)(|x-y|+|t-s|).$$
        \item\label{A4} For every $R > 0$, there is a constant $A_R$ such that if $p,q\in B_R(0)$ then
           $$|H(t,x,u,p)-H(t,x,u,q)|\le A_R|p-q|,$$ uniformly in $(t,x,u)$. 
           \item \label{A5}There is a positive constant $K_0$ such that 
           $$||u_0||_{W^{1,\infty}(\R^n)} \le K_0.$$

    \end{Assumptions}

       \begin{rem}

   The assumptions \ref{A1}-\ref{A5} are natural and standard, except perhaps \ref{A2} where the Hamiltonian $H$ is assumed to be linear in $`u'$. Ideally, if $H$ is only monotonically increasing in $u$ then  the initial value problem \eqref{eq:HJB-eq} is well-posed.  However, we are interested to investigate the regularizing effect of $\frac{1}{2}$-Laplacian on this problem and the assumption \ref{A2} will be necessary for our methodology to work.
  \end{rem}

     We now define the notion of viscosity solution for nonlocal equations of type \eqref{eq:HJB-eq-viscous}. We point out that there could be more than one ways to formulate the definition of sub-/supersolution of the equation \eqref{eq:HJB-eq-viscous}, but various apparently different formulations lead to the same notion.  We use the formulation from \cite{Jakobsen:2005jy} to define the sub- and supersolutions. To this end, we introduce the following quantities. For $\kappa\in (0,1)$, let
     \begin{align*}
      \mathcal{I}_\kappa^\epsilon(\varphi) = -\epsilon C(n, 1) \int_{B(0,\kappa)}\frac{\big(\varphi(t,x+z)-\varphi(t,x)\big)}{|z|^{(n+1)}} dz,\\
        \mathcal{I}^{\kappa,\epsilon}(u) = -\epsilon C(n,1) \int_{B(0,\kappa)^{c}}\frac{\big(u(t,x+z)-u(t,x)\big)}{|z|^{(n+1)}} dz.
     \end{align*}

        By the representation \eqref{eq:non-local-representaion}, for any $\kappa$,  one can  rewrite $\epsilon (-\Delta)^{\frac{1}{2}}\varphi$ as
    \begin{align*}
        \epsilon (-\Delta)^{\frac{1}{2}}\varphi =    \mathcal{I}_\kappa^\epsilon(\varphi)+   \mathcal{I}^{\kappa,\epsilon}(\varphi)
    \end{align*} and define viscosity solutions as follows.
     
         \begin{defi}[viscosity solution]\label{defi:viscosolution}
      \begin{itemize}
        \item[$i.)$] A function $u\in USC_b([a,b]\times \mathbb{R}^n)$ is a viscosity subsolution of  \eqref{eq:HJB-eq-viscous-1} if for any $\varphi\in C^{1,2}([a,b]\times \mathbb{R}^n)$, whenever $(t,x)\in(a,b)\times \mathbb{R}^n$ is a global maximum point of $u-\varphi$ it holds that
        \begin{align*}
        \varphi_t(t,x)+H(t,x, u(t,x),\nabla \varphi(t,x)) +   \mathcal{I}_\kappa^\epsilon(\varphi)+   \mathcal{I}^{\kappa,\epsilon}(u) \le 0
        \end{align*}  for all $\kappa\in (0,1)$.
          \item[$ii.)$] A function $u\in LSC_b([a,b]\times \mathbb{R}^n)$ is a viscosity supersolution of  \eqref{eq:HJB-eq-viscous-1} if for any $\varphi\in C^{1,2}([a,b]\times \mathbb{R}^n)$, whenever $(t,x)\in(a,b)\times \mathbb{R}^n$ is a global minimum point of $u-\varphi$, it holds that
        \begin{align*}
        \varphi_t(t,x)+H(t,x, u(t,x),\nabla \varphi(t,x)) +   \mathcal{I}_\kappa^\epsilon(\varphi)+   \mathcal{I}^{\kappa,\epsilon}(u) \ge 0
        \end{align*}
        for all $\kappa\in (0,1)$.
           \item[$iii.)$] A function $u\in C_b([a,b]\times \mathbb{R}^n)$ is a viscosity solution of \eqref{eq:HJB-eq-viscous-1} if it is both a sub and supersolution.
     \end{itemize}
    \end{defi}

      The Definition \ref{defi:viscosolution} is also applicable to the case $\epsilon =0$ i.e. when the fractional diffusion term is absent. In this case however, the condition $\kappa \in (0,1)$ becomes redundant. Note that the test function appears  in the nonlocal part of  \eqref{eq:HJB-eq-viscous} and this is unavoidable due to the singular nature of the weight function $|z|^{-(n+1)}$ in \eqref{eq:non-local-representaion}. Some growth assumptions are needed on the sub and supersolutions for the nonlocal term $ \mathcal{I}^{\kappa,\epsilon}(u)$ to be finite;  boundedness assumption is not the most general but sufficient  for our framework.

    As usual, any classical solution is also a viscosity solution and any smooth viscosity solution is a classical solution. Furthermore, an equivalent definition is obtained by replacing ``global maximum/minimum" by  ``strict global maximum/minimum" in the above definition.   We may also assume $\varphi = u$ at the maximum/minimum point. Next, we give an alternative (equivalent) definition which will be used to prove the existence of viscosity solutions via Perron's method.

\begin{lem}[alternative definition]\label{lem:alt-defi}
   A function $v\in USC_b([a,b]\times \mathbb{R}^n)$ (or $v\in LSC_b([a,b]\times \mathbb{R}^n)$ ) is a viscosity subsolution (supersolution) to  \eqref{eq:HJB-eq-viscous-1} iff for every $(t,x)\in(a,b)\times \mathbb{R}^n $ and $\phi \in C^{1,2}_0([a,b]\times \mathbb{R}^n)$ such that $v-\phi$ has a global maximum (minimum) at $(t,x)$ then
   \begin{align*}
    \phi_t(t,x)+H(t,x, u(t,x), \nabla \phi(t,x)) +   \epsilon (-\Delta)^{\frac{1}{2}}\phi   \le 0 ~~(\ge 0).
   \end{align*}
\end{lem}
     We refer to \cite{Jakobsen:2005jy} for a proof.

\begin{rem} The definition of viscosity solution is not influenced by the choice of $(0,1)$ as the domain for the parameter $\kappa$. Equivalently, one can replace $(0,1)$ by an interval of type $(0,\delta)$ for $\delta > 0$. All these different choices for domain of $\kappa$ could be proven to be equivalent to alternative definition in Lemma \ref{lem:alt-defi}.  However, in order for our methodology to work, we need to be able to pass to the limit $\kappa \rightarrow 0$ and Definition \ref{defi:viscosolution} is formulated keeping that in mind.
\end{rem}

          For $\ell\in\mathbb{R}^n$, define $H_\ell:  \R\times\R^n\times\R\times \R^n\rightarrow \mathbb{R} $ as
         \begin{align}\label{eq:defi_perturbation}
          H_{\ell}(t,x,r,p) = H(t,x+\ell, r, p)
         \end{align} and consider the  perturbation 
 \begin{align}
 \label{eq:HJB-eq-viscous-pert}v_t + H_{\ell}(t,x, v(t,x), \nabla v(t,x))+\eps (-\Delta )^{\frac{1}{2}}v(t,x) &= 0\quad\text{if}\quad (t,x) \in (0,T]\times\mathbb{R}^n
 \end{align} of \eqref{eq:HJB-eq-viscous-1}.
 We have the following continuous dependence estimate.

 \begin{thm}[continuous dependence]\label{thm:continuous dependence}
  Assume \ref{A1}-\ref{A4}, and let $u, -v\in USC_b([0,T]\times \R^n)$   respectively satisfy
  \begin{align}
  u_t + H(t,x, u, \nabla u)+\eps (-\Delta )^{\frac{s}{2}}u &\le0,\\
  v_t + H_{\ell}(t,x, v, \nabla v)+\eps (-\Delta )^{\frac{s}{2}}v &\ge 0
  \end{align} in the viscosity sense\footnote{ This simply means that $u$ and $v$ are respectively the sub and supersolution of \eqref{eq:HJB-eq-viscous-1} and \eqref{eq:HJB-eq-viscous-pert}}. Also assume $|\nabla u(0,x)|+|\nabla v(0,x)| \le C^\prime$ for some constant $C^\prime$. Then there exists a constant $C$, depending on the data (including $T$ but excluding $\epsilon$), such that
  \begin{align*}
  v-u \le (v(0,x)-u(0,x))^++ C |\ell|.
  \end{align*}
  \end{thm}
 \begin{proof}
      This theorem is a special case of much more general results by Jakobsen $\&$ Karlsen \cite{Jakobsen:2005jy}, we refer to this article for a detailed proof.
 \end{proof}

   The standard comparison principle for sub and super-solutions of \eqref{eq:HJB-eq-viscous} follows  immediately as a consequence of Theorem \ref{thm:continuous dependence}  if we choose $\ell = 0$.

 \begin{cor}[comparison principle]
  Let the assumptions \ref{A1}-\ref{A4} be true and  $u, -v\in USC_b([0,T]\times \R^n)$   respectively satisfy
  \begin{align}
  u_t + H(t,x, u, \nabla u)+\eps (-\Delta )^{\frac{s}{2}}u &\le0,\\
  v_t + H(t,x, v, \nabla v)+\eps (-\Delta )^{\frac{s}{2}}v &\ge 0
  \end{align} in the viscosity sense. Furthermore,  assume that $|\nabla u(0,x)|+|\nabla v(0,x)| \le C$ for some constant $C$. In addition, if $u(0,x)\le v(0,x)$  then
  $$u\le v\quad \text{in}\quad (0,T]\times \R^n. $$
 \end{cor}
 The comparison principle ensures the uniqueness of viscosity solution for the initial value problem  \eqref{eq:HJB-eq-viscous-1}. The proof for existence uses the standard Perron's method for viscosity solution framework. However, our definition of viscosity solution is slightly different compared to \cite{Imbert: 2005ft} and we provide detailed proof of existence.

 \begin{thm}[existence]\label{thm:existence}
  Assume \ref{A1}-\ref{A5}. There exists unique viscosity solution $u\in C_b([0,T]\times \R^n)$ of the initial value problem \eqref{eq:HJB-eq-viscous-1}.
 \end{thm}

 \begin{proof}
 We begin with the claim that without loss of generality we may assume $u_0\in C_b^2$.
 \vspace{.5cm}

 \noindent{\em Justification:} Suppose that we have proven the existence of a viscosity solution for $C_b^2$ initial data i.e $u_0\in C_b^2(\R^n)$.  Now, if $u_0\in W^{1,\infty}(\R^n)$, then there exists a sequence $(u_0^k)_k$ such that $u_0^k\in C_b^2(\R^n)$ and $u_0^k \rightarrow u_0$  uniformly to $u_0$ as $k\rightarrow \infty$. Let $u^m(t,x)$ be the solution of the equation  \eqref{eq:HJB-eq-viscous-1} with initial condition $u^m(0,x) = u_0^m$. Then by the continuous dependence estimate

 \begin{align}
 ||u^m(t,x)- u^p(t,x)||_{L^{\infty}([0,T]\times \R^n)} \le C ||u_0^m-u_0^p||_{L^\infty(\R^n)}.
 \end{align} Therefore  the sequence $(u^m(t,x))$ is Cauchy in $C_0([0,T]\times \R^n)$, and it will converge to some function $u(t,x)\in C_0([0,T]\times \R^n)$.  We now use the stability property of viscosity solutions and conclude that $u(t,x)$ is a viscosity solution of  \eqref{eq:HJB-eq-viscous-1} with $u(0,x) = u_0(x)$.

    Now we prove the existence of a viscosity solution for $C_b^2$-initial condition.  Define
    \begin{align}
    \overline{u}(t,x) = u_0(x) + Ct,\\
    \underline{u}(t,x) = u_0(x) - Ct.
    \end{align} Invoke the assumptions \ref{A1}-\ref{A5} and choose the constant $C$ big enough such that
    \begin{align*}
     C\ge |H(t,x, u_0(x), \nabla u_0(x))| + \epsilon|(-\Delta)^{\frac{1}{2}} u_0(x)|.
    \end{align*} Then the functions $\underline{u}(t,x)$ and $\overline{u}(t,x)$ are respectively a sub and supersolution of \eqref{eq:HJB-eq-viscous-1} satisfying the initial condition $u(0,x) = u_0(x)$.

  Define $v(t,x)$ as
  \begin{align*}
   v(t,x) = \sup \big\{w(t,x): w\le \overline{u}(t,x) ~\text{and} ~w~ \text{is subsolution of  \eqref{eq:HJB-eq-viscous-1} satisfying the initial condition.}\big\}
  \end{align*}

 Next, let $v_* $ and $v^*$ denote the upper and lower semicontinuous envelopes of $v(t,x)$:
  \begin{align*}
   v_*(t,x) = \lim_{r\downarrow 0}\sup \big\{ v(s,y): (s,y) \in B_r(t,x)\cap [0, T)\times \R^n \big\}
  \end{align*} and $v^*(t,x)= -(-v)_*(t,x)$. From the definition it is clear that
  \begin{align*}
   \underline{u}\le v_*, \quad v^* \le \overline{u}\quad \text{and}\quad v^* \le v_*.
  \end{align*}
    The functions $\overline{u}, ~\underline{u}$ are uniformly continuous, and hence we must have

    \begin{align}
    \underline{u} = (\underline{u})^* \le v^* \le v_* \le (\overline{u})_*= \overline{u}.
    \end{align} Therefore $v_*(0,x) = v^*(0,x) = u_0(x)$. We want to show that $v_*$ and $v^*$ are respectively sub- and supersolution of \eqref{eq:HJB-eq-viscous-1}. This will be enough to ensure the existence, since by the comparison principle
    \begin{align*}
     v_* \le v^*
    \end{align*}  and hence $v_* = v^* = v$ is the sought after (continuous ) viscosity solution of   \eqref{eq:HJB-eq-viscous-1} satisfying the initial condition $v(0,x) = u_0(x)$.

  It is relatively straightforward to show that $ v_*$ is a subsolution and the details are as follows. For every $(t,x) \in (0,T)\times \R^n$, there is a sequence $(t_p, x_p, u_p(t_p, x_p))_p$ such that
  \begin{align*}
  \lim_{p\rightarrow \infty} \big(t_p, x_p, u_p(t_p, x_p)\big) = (t,x, v_*(t,x)),
  \end{align*} where $u_p$ is a subsolution for each $p\in \mathbb{N}$.  Now for some $\phi\in C^{1,2}$, if $v_*-\phi$ has the strict global maximum at $(t,x)$, then there will be sequence $(s_p, y_p)_p$ such that $u_p-\phi$ will have global maximum at $(s_p, y_p)$ and
  \begin{align*}
  \lim_{p\rightarrow \infty}\big( s_p, y_p, u_p(s_p, y_p)\big) = \big(t,x, v_*(t,x)\big).
  \end{align*} Furthermore, $s_p> 0$ for $p$ large enough and from the definition of subsolution we obtain

   \begin{align}
   \label{eq:sub-sol-final} \phi_t(s_p,y_p)+H(s_p,y_p, u_p(s_p,y_p),\nabla \phi(s_p,y_p)) +   \epsilon (-\Delta)^{\frac{1}{2}}\phi(s_p, y_p)\le  0.
   \end{align} Finally, we use the continuity of the equation and pass to the limit $p\rightarrow \infty$ in \eqref{eq:sub-sol-final} and conclude that $v_* $ is subslolution of \eqref{eq:HJB-eq-viscous}.

        Next, we prove that $v^*$ is a supersolution.  We employ the method of contradiction and assume that there exists $(t,x)\in (0,T)\times \R^n$, $\phi\in C^{1,2}_b$ satisfying $v^*(t,x) = \phi(t,x)$ and $v^*-\phi$ has a global minimum at $(t,x)$ such that
        \begin{align}
       \label{eq:sup-sol-contradic}  \phi_t(t,x)+H(t,x, v^*(t,x),\nabla \phi(t,x)) +   \epsilon (-\Delta)^{\frac{1}{2}}\phi< 0.
        \end{align}     From the definition it follows that $v^*(t,x)\le \overline{u}(t,x)$. We claim however  that $ v^*(t,x) < \overline{u}(t,x)$. Otherwise, $\phi(t,x) = \overline{u}(t,x) = v^*(t,x)$ and $\bar{u}-\phi$ will have a global minimum at $(t,x)$ and

  \begin{align*}
       \phi_t(t,x)+H(t,x, v^*(t,x),\nabla \phi(t,x)) +   \epsilon (-\Delta)^{\frac{1}{2}}\phi(t,x)\ge 0,
        \end{align*} which contradicts \eqref{eq:sup-sol-contradic}.

  By the continuity of $\phi$ and $\overline{u}$,  there are constants $\gamma_1, \delta_1 > 0$ such that
  \begin{align*}
  \phi +\gamma_1 \le \overline{u}\quad\text{in}\quad B_{\delta_1}(t,x)\subset (0,T)\times \R^n.
  \end{align*}
        Moreover, by \eqref{eq:sup-sol-contradic} and continuity of the equation, there exist two constants $\gamma_2, \delta _2 > 0$ such that

      \begin{align}
       \label{eq:sup:cond}(\phi+ \gamma)_t(s,y)+H(s,y, (\phi+\gamma)(s,y),\nabla(\phi+\gamma)) +   \epsilon (-\Delta)^{\frac{1}{2}}\phi(s,y)\le 0
        \end{align} for all $(s,y) \in B_{\delta_2}(t,x)$ and $0<\gamma\le  \gamma_2$. Sine $v^*-\phi$ has a strict minimum at $(t,x)$, there are constants $\gamma_3 $ and $0<\delta_0\le \min(\delta_1,\delta_2)$ such that $v^*-\phi > \gamma_3$ on $\partial B_{\delta_0}(t,x)$.  Now set $\gamma_0 = \min(\gamma_1, \gamma_2, \gamma_3)$ and define

        \begin{align*}
         w = \begin{cases} \max(\phi+\gamma_0, v_*)\quad \text{on}\quad B_{\delta_0}(t,x)\cap [0,T]\times \R^n\\
                              v_*\quad \quad\text{otherwise}.
        \end{cases}
        \end{align*}

  Note that $w$ is upper semicontinuous. We argue that $w$ is a subsolution of \eqref{eq:HJB-eq-viscous-1}. Let $(s,y)\in (0,T)\times \R^n$ and $\psi\in C_b^{1,2}$ be test function such that $\psi(s,y) = w(s,y)$ and $w-\psi$ has strict global maximum at $(s,y)$. Depending on whether $w = v_*$ or $w = \phi +\gamma_0$ at $(s,y)$,  either $ v_*-\psi$ or $\phi+\gamma_0 -\psi$ has a global maximum at $(s,y)$.  In the first case, the subsolution inequality for $\phi$ is a consequence of $v_*$ being a subsolution. In the other case
    \begin{align*}
    \partial_t \phi(s,y) \ge \partial_t\psi(s,y),\quad \nabla \phi(s,y) = \nabla \psi(s,y),\quad  \mathcal{I}_\kappa^\epsilon(\phi)+   \mathcal{I}^{\kappa,\epsilon}(\phi)\ge  \mathcal{I}_\kappa^\epsilon(\psi)+   \mathcal{I}^{\kappa,\epsilon}(\psi)
    \end{align*} and hence by \eqref{eq:sup:cond}

   \begin{align*}
   \psi_t(t,x)+H(t,x,w(t,x),\nabla \psi(t,x)) +   \epsilon (-\Delta)^{\frac{1}{2}}\psi\le 0.
        \end{align*}  Therefore $w$ is a subsolution of \eqref{eq:HJB-eq-viscous-1} satisfying the initial condition.  At the point $(t,x)$, we have

        $$w^*(t,x) \ge \max\{ \phi(t,x)+\gamma_0, v^*(t,x)\} = \phi(t,x)+\gamma_0 = v^*(t,x)+ \gamma_0$$ i.e., $w(s,y) > v(s,y)$ for some $(s,y)$, which contradicts the definition of $v$.

\end{proof}

It follows as  a simple consequence of the continuous dependence estimate that  the unique viscosity solution of \eqref{eq:HJB-eq-viscous-1} is Lipschitz continuous.

 \begin{lem}[Lipschitz continuity]\label{lip}
  Assume that \ref{A1}-\ref{A5} hold, and let $u^\epsilon\in C_b$ be the unique viscosity solution of \eqref{eq:HJB-eq-viscous-1} with $0\le \epsilon \le 1$.  Then there exists a constant $L$, depending on the data ( including $T$ but excluding $\epsilon$) such that

  \begin{align*}
  |u^\epsilon(t,x+h)- u^\epsilon(t,x)| \le L |h|
  \end{align*} for all $h\in \R^n$ and $(t,x)\in [0,T)\times \R^n$.

  \end{lem}

  \begin{proof}
   Obviously, the function $v = u^\epsilon(t,x+h)$ is the unique viscosity solution of \eqref{eq:HJB-eq-viscous-pert} with ($\ell = h$ and) initial condition $ v(0,x) = u_0(x+h)$. Hence the proof follows once we apply Theorem \ref{thm:continuous dependence}.
  \end{proof}	

  Our primary aim in this article is to  extend the $C^{1,\alpha}$-regularity results in \cite{silvestre:2010} to the problem \eqref{eq:HJB-eq-viscous-1}. To this end, it is noteworthy that the following identity holds: (cf. \cite{silvestre:2010})
  \begin{align*}
  (-\Delta)^{\frac 12} u = \sum_{i=1}^{n} \mathcal{R}_i \partial_{x_i} u,
  \end{align*}where $\mathcal{R}_i$s are the classical Reisz transforms in $\R^n$.  For  a function $u\in C^{1,\alpha}(\R^n)$,  it follows from the classical $C^\alpha $ estimates for the Reisz transforms  that  $(-\Delta)^{\frac{1}{2}}u $ is $ C^\alpha$.  In other words, if a solution to \eqref{eq:HJB-eq-viscous-1} is proven to be $C^{1,\alpha}$, it will be a classical solution. A more precise mathematical formalization of these facts is given as the next proposition, a detailed proof of which can be found in \cite{silvestre:2010}.

  \begin{prop}\label{prop:frac_continuity}
   Given $u\in C^{1,\alpha}$ and the integro-differential  operator
   \begin{align*}
   \mathcal{L} u = \int_{\mathbb{R}^n} \frac{u(x+y)-u(x)}{|y|^{n+1}}dy;
   \end{align*} the function $\mathcal{L}u$ is a $C^\alpha$ function and $C^\alpha$-norm depends on $||u||_{C^{1,\alpha}}$ and the dimension $n$.

   As a result, under the assumptions \ref{A1}-\ref{A4}, the nonlinear operator $\mathcal{I}u = u_t+ H(t,x,u,\nabla u) + \epsilon (-\Delta)^{\frac{1}{2}} u$ also maps $u$ to a $C^\alpha$ function.

  \end{prop}

 \subsection{The main results}
   As has  been mentioned a few times already, part of our main goal in this article is to extend the $C^{1,\alpha}$-type regularity estimate for  \eqref{eq:HJB-eq-viscous-1}. Besides, we establish that the solutions of  \eqref{eq:HJB-eq-viscous-1} converges to the unique viscosity solution of  \eqref{eq:HJB-eq} as $\epsilon \rightarrow 0$ and estimate the rate of convergence.

 \begin{thm}[$C^{1,\alpha}-regularity$]\label{thm:regularity_theorem}
  Assume \ref{A1}-\ref{A5}, and let $u^\epsilon$ be the unique viscosity solution of  \eqref{eq:HJB-eq-viscous-1}. There exists a positive constant $\alpha$ , depending on $K, n, \epsilon$ and $T$, such that for every $t>0$ and $x\in \R^n$,  $u$ is $C^{1,\alpha}$  at $(t,x)$.  Moreover

  \begin{align*}
     ||u^\epsilon||_{C^{1,\alpha}((\frac t2, t]\times \R^n)}\le \frac{C}{t^\alpha} \big( K + ||\nabla u_0(x)||_{L^\infty}\big).
  \end{align*}
 \end{thm}

 \begin{thm}[convergence rate]\label{thm:vanishing viscosity rate}
  Assume \ref{A1}-\ref{A5}. For $\epsilon \in (0,e^{-1})$, let $u^\epsilon$ and $u$ be respectively  the unique viscosity solutions of \eqref{eq:HJB-eq-viscous-1} and \eqref{eq:HJB-eq} . Then there exists a constant $C$ depending on the data (not on $\epsilon$) such that

  \begin{align*}
   |u^\epsilon(t,x)- u(t,x)|| \le C \epsilon|\log \epsilon|
  \end{align*} for all $(t,x)\in [0,T]\times \R^n$.
 \end{thm}
\begin{rem}
    It was brought to our attention by the referee that Theorem \ref{thm:vanishing viscosity rate} is not new, and the very same results on error estimates related to critical fractional Laplacian have been obtained earlier by Droniou and Imbert \cite{Dro2006}. Also, the proof by Droniou and Imbert \cite{Dro2006} does not require any additional regularity on the unknowns other than the Lipschitz continuity. 
\end{rem}
 In the remaining part this section, let us outline the recipe to prove the regularity estimate in Theorem \ref{thm:regularity_theorem} and prove a couple of technical lemmas in connection to this. Let $\ell\in \R^n$ be a unit vector and assume for a moment that the solution $u^\eps$ of \eqref{eq:HJB-eq-viscous-1} is smooth.  Then the directional derivative $v = \partial_\ell u^\eps$ would satisfy the following linearized equation
\begin{align}\label{eq:linearization}
v_t + D_p H(t,x, u^\eps, \nabla u^\eps).\nabla v + D_x H(t,x,u^\eps,\nabla u).\ell + \partial_u H(t,x,u^\eps,\nabla u^\eps)v +\epsilon(-\Delta)^{\frac{1}{2}}v = 0.
\end{align} In other words,  $v$ satisfies the following fractional advection-diffusion equation:
\begin{align} \label{eq:advection-diffusion-source}
v_t + w.\nabla v + \lambda v + f(t,x) + \epsilon(-\Delta)^{\frac{1}{2}}v  = 0
\end{align} where $ w =D_p H(t,x,u^\eps,\nabla u^\eps),~ f(t,x) = D_xH(t,x,\nabla u^\eps).\ell $
 and we have used the assumption that $H$ is linear in $u$. Keeping in  line with \cite{silvestre:2010},  the idea is to  get a $C^\alpha$-type estimate for \eqref{eq:advection-diffusion-source} and translate that into a $C^{1,\alpha}$-type estimate for \eqref{eq:HJB-eq-viscous-1}.

   The problem of proving H\"{o}lder continuity for critical fractional advection-diffusion equations like \eqref{eq:advection-diffusion-source} is a delicate one,  but the recent works of Caffarelli $\&$ Vasseur \cite{cafarelli:2010} and Silvestre\cite{silvestre:2010}  have contributed greatly to the understanding of this problem.  In the situation when $f(t,x) =\lambda = 0$ and $\text{div}~ w = 0$,  Caffarelli $\&$ Vasseur showed that the weak solutions to \eqref{eq:advection-diffusion-source} becomes holder continuous for positive time if $w\in \text{VMO}$. In \cite{cafarelli:2010}, the authors use variational techniques and follow the Di-Giorgi type approach. In a subsequent development, Silvestre\cite{silvestre:2010} uses viscosity solution approach  and proves similar H\"{o}lder continuity estimates  under the only assumption that $w\in L^\infty$.  In this article we employ the later approach.

   Remember  that a priori we do not have any information on $w$ and $f(t,x)$ except that they are bounded. Also, the vector field $w$ may not be divergence free. This makes the weak formulation by means of integration by parts unfeasible.  However, it is possible to make sense of the inequalities
   \begin{align}
   \label{eq:diff-convec-ineq-1} v_t-A|\nabla v| - B + \lambda v +\epsilon(-\Delta)^{\frac{1}{2}}v &\le 0\\
   \label{eq:diff-convec-ineq-2} v_t + A|\nabla v| + B+ \lambda v + \epsilon(-\Delta)^{\frac{1}{2}}v  &\ge 0
   \end{align} in the viscosity sense.  If we invoke \ref{A1}-\ref{A4} and select $A \ge  \sup_{t,x, |u|\le ||u^\epsilon||_\infty, |p|\le L }|D_pH(t,x, u,\cdot)|$ ($L$ is from Lemma \ref{lip}) and $B \ge C(1+  ||u^\epsilon||_{L^\infty})$ for all $\epsilon$, then the inequalities \eqref{eq:diff-convec-ineq-1}-\eqref{eq:diff-convec-ineq-2} are in perfect correspondence with \eqref{eq:advection-diffusion-source}. Our strategy is to follow and extend the methodology of \cite{silvestre:2010} and establish H\"{o}lder continuity for functions satisfying  \eqref{eq:diff-convec-ineq-1}-\eqref{eq:diff-convec-ineq-2} and then translate it properly to $C^{1,\alpha}$ estimate for \eqref{eq:HJB-eq-viscous-1}.

 It is needless to mention that viscosity solutions are not a priori smooth enough to undergo above operations. However, one can formally justify that the finite difference quotients would also satisfy the inequalities  \eqref{eq:diff-convec-ineq-1}-\eqref{eq:diff-convec-ineq-2}, and the plan is to use this information to establish a uniform $C^\alpha$-estimate for the finite difference quotients. The next lemma makes this connection rigorous.

  \begin{lem}\label{lemma: equation for the finte difference}
  Assume \ref{A1}-\ref{A4}, and let $u, -v\in USC_b([0,T]\times \R^n)$   respectively satisfy
  \begin{align}
  u_t + H_\ell(t,x, u,\nabla u)+\eps (-\Delta )^{\frac{1}{2}}u &\le0,\\
  v_t + H(t,x, v, \nabla v)+\eps (-\Delta )^{\frac{1}{2}}v &\ge 0
  \end{align} in the viscosity sense.  Furthermore, assume that there is a constant $K$ such that\newline  $||\nabla_x u(t,\cdot)||_{L^\infty(\R^n)} +||\nabla_x v(t,\cdot)||_{L^\infty(\R^n)} \le K$ for all $t\in [0,T]$. Then $(u-v)$ satisfies
  \begin{align}
  \label{eq:conclu_perturb_1}(u-v)_t -A|\nabla(u-v)| -B|\ell| +\lambda (u-v)+ \epsilon(-\Delta)^{\frac{1}{2}}(u-v)\le 0
  \end{align} in the viscosity sense.
  On the other hand, if $u, -v\in LSC_b([0,T]\times \R^n)$   respectively satisfy
  \begin{align}
  u_t + H_\ell(t,x, u, \nabla u)+\eps (-\Delta )^{\frac{1}{2}}u &\ge0,\\
  v_t + H(t,x, v, \nabla v)+\eps (-\Delta )^{\frac{1}{2}}v &\le 0
  \end{align} in the viscosity sense, and there is a constant $K$ such that $||\nabla_x u(t,\cdot)||_{L^\infty(\R^n)} +||\nabla_x v(t,\cdot)||_{L^\infty(\R^n)} \le K$ for all $t\in [0,T]$. Then $(u-v)$ satisfies
  \begin{align}
  \label{eq:conclu_perturb_2}  (u-v)_t +A|\nabla(u-v)| +B|\ell| +\lambda (u-v)+ \epsilon(-\Delta)^{\frac{1}{2}}(u-v)\ge 0
  \end{align}in the viscosity sense.

 \end{lem}
\begin{rem}
   It is to be noted in the above statement that we require the sub and supersolutions  ($u$ and $v $)  to be Lipschtiz continuous in $x$, the space variable. In our scheme of work,  we will apply Lemma \ref{lemma: equation for the finte difference} where $u$ and $v$ are Lipschitz continuous viscosity solutions of \eqref{eq:HJB-eq-viscous-pert} and \eqref{eq:HJB-eq-viscous-1}, respectively. 
\end{rem}

        Before the details of the proof could be furnished, we need to introduce the notion of sup/inf convolution. 


 \begin{defi}[sup/inf convolutions]
 Given an upper semicontinuous function $u(t,x)$ and positive constants $\delta >0$, the sup-convolution $u^\delta$ is defined as
 \begin{align}
  u^\delta(t,x) = \sup_{y\in \R^n, s\in [0,\infty)} \Big[ u(y,s)-\frac 1\delta \big(|x-y|^2+(t-s)^2\big)\Big].
 \end{align} Similarly, for a lower semicontinuous function $v$, the inf-convolution $v_\delta$ is defined as
  \begin{align}
  v_\delta(t,x) = \inf_{y\in \R^n, s\in [0,\infty)} \Big[ v(y,s)+\frac 1\delta \big(|x-y|^2+(t-s)^2\big)\Big].
 \end{align}
 \end{defi}

 \begin{rem} For any upper or lower semicontinuous function on $[0,T]\times \R^n$,  we first trivially extend the function on $[0,\infty)\times \R^n$ and then define the respective sup or inf-convolution.
 \end{rem}
   We have the following lemma, the proof of which is built on ideas borrowed from \cite{Ishii:1995gs}.

 \begin{lem}\label{lem:semiconvex-concave}
 Let $u(t,x)\in USC_b([0, T ] \times \R^n) $ and $v(t,x)\in LSC_b([0, T ] \times \R^n) $  are respectively sub and supersolution of \eqref{eq:HJB-eq-viscous-1}.  Furthermore, assume that there is a constant $K$ such that $||\nabla_x u(t,\cdot)||_{L^\infty(\R^n)} +||\nabla_x v(t,\cdot)||_{L^\infty(\R^n)} \le K$ for all $t\in [0,T]$. Then for every $\vartheta > 0$, there exists a $\delta_0 > 0$ such that for all $0<\delta <\delta_0$, it holds in the viscosity sense that
   \begin{align}
    \label{eq:sup-conv-sub-solution}(u^\delta)_t + H(t,x,u^\delta,\nabla u^\delta)+\epsilon(-\Delta)^{\frac{1}{2}}u^\delta \le \vartheta.
   \end{align}
 Also,
 for all $0<\delta <\delta_0$, it holds in the viscosity sense that
   \begin{align}
   \label{eq:sup-conv-sup-solution}   (v_\delta)_t + H(t,x,v_\delta,\nabla v_\delta)+\epsilon(-\Delta)^{\frac{1}{2}}v_\delta \ge -\vartheta.
   \end{align}

 \end{lem}

 \begin{proof}
  We will provide a detailed proof for the first half i.e. \eqref{eq:sup-conv-sub-solution}; the proof of the second half is similar. 

 Choose $M> 0$ such that $M \ge 2 \sup |u|$. For any $\delta > 0$, obviously $u\le u^\delta$ in $[0,T]\times \R^n$. Therefore it is easily seen that if $\gamma =(\delta_0 M)^{\frac 12}$, then
 \begin{align*}
    u^\delta(t,x) &= \sup\Big\{u(s,y) - \frac 1\delta \big(|x-y|^2+(t-s)^2\big): |x-y|^2+|t-s|^2 <\gamma\Big\}\\
                          & =  \sup\Big\{u(t+s,x+y) - \frac 1\delta \big(|y|^2+(s)^2\big): |y|^2+|s|^2 <\gamma\Big\}.
 \end{align*}
     Clearly $u(\cdot+s, \cdot+y)$ is a subsolution of
         \begin{align}\label{eq:translated-equation}
         \hat{u}_t + H(t+s, x+y, \hat{u},\nabla \hat{u}) + \epsilon(-\Delta)^{\frac{1}{2}} \hat{u} = 0.
         \end{align}

 By \ref{A2}, $H$ is linear in $u$ and $\partial_u F(t,x,u,p)$ is a nonnegative constant. Therefore $u^{\prime}(\cdot, \cdot)=u(\cdot+s, \cdot+y)-\frac 1\delta( |y|^2+|s|^2)$ is also a subsolution of
 \eqref{eq:translated-equation}. Note that, just as $u$, $u^\prime(t,x)$ is also Lipschitz continuous in $x$ and $||\nabla_x u^\prime(t,\cdot)||_{L^\infty(\R^n)} \le K$ for all $t\in [0,T]$. Now if $\phi$ is a test function such that $u-\phi$ has a global maximum at $(t,x)$, then it holds that $|\nabla_x \phi(t,x)| \le ||\nabla_x u^\prime(t,\cdot)||_{L^\infty(\R^n)} \le K$. Therefore, we invoke \ref{A3} -\ref{A4} along with the Lipschitz continuity of $u(t,\cdot)$ and conclude that $u^\prime = u(\cdot+s, \cdot+y)-\frac 1\delta( |y|^2+|s|^2)$ satisfies
 \begin{align*}
  u^\prime_t + H(t, x, u^\prime,\nabla u^\prime) + \epsilon(-\Delta)^{\frac{1}{2}} u^\prime &\le C(|s|^2+|y|^2)^{\frac 12}\\
                                                              & \le C\gamma
 \end{align*} in the viscosity sense. Choose $\delta_0$, so that, $C \gamma <\vartheta$. Then it holds in the viscosity sense it holds that

 \begin{align}
  \label{eq:subsoll-perturbed}u^\prime_t + H(t, x, u^\prime,\nabla u^\prime) + \epsilon(-\Delta)^{\frac{1}{2}} u^\prime &\le \vartheta
  \end{align}
  We now take the supremum and argue as in Theorem \ref{thm:existence} to conclude that $u^\delta =  \sup\Big\{u(t+s,x+y) - \frac 1\delta \big(|y|^2+|s|^2\big): |y|^2+|s|^2 <\gamma\Big\}$ satisfies \eqref{eq:subsoll-perturbed} in the viscosity sense.

 \end{proof}

\begin{proof}[Proof of Lemma \ref{lemma: equation for the finte difference}]
    By \cite[Proposition A.5]{silvestre:2010}, $u^\delta \rightarrow u$ and $v_\delta \rightarrow v$ in the half relaxed sense as $\delta \rightarrow 0$.  To prove the first part, it would be enough to show
        \begin{align}\label{eq:eq_for_perturbed_fd}
  (u-v)_t -A|\nabla(u-v)| -B|\ell| +\lambda (u-v)+ \epsilon(-\Delta)^{\frac{1}{2}}(u-v)\le 2\vartheta
  \end{align} in the viscosity sense for all $\vartheta > 0$. Hence, by the stability of viscosity solutions under half relaxed limit it would suffice if we prove $(u^\delta -v_\delta) $ satisfies \eqref{eq:eq_for_perturbed_fd} for small enough $\delta$'s.  In view of Lemma \ref{lem:semiconvex-concave}, the rest of the argument is same as \cite[Lemma 3.2]{silvestre:2010}.

   Let $\varphi$ be a test function which touches $u^\delta-v_\delta$ at $(t,x)\in (0,T]\times \R^n$ from above. For any $\delta >0$, $u^\delta$ and $-v^\delta$ are semi-convex functions, which means they have tangent paraboloid from below of opening $\frac 1\delta$.   Since the test function $\varphi$ touches $u^\delta-v_\delta$ from above at $(t,x)$, both $u^\delta$ and $v_\delta$ must be $C^{1,1}$ at $(t,x)$. Therefore, it is implied that $ \partial_t u^\delta, \partial_t v_\delta , \nabla u^\delta, \nabla v_\delta$ are well defined at $(t,x)$. It is also implied that $(-\Delta )^{\frac{1}{2}} u^\delta $ and $(-\Delta )^{\frac{1}{2}} v_\delta$ are well defined at $(t,x)$.  At the point $(t,x)$, it follows by direct computation that

     \begin{align}
  \nonumber&(u^\delta-v_\delta)_t -A|\nabla(u^\delta-v_\delta)| -B|\ell| +\lambda (u^\delta-v_\delta)+ \epsilon(-\Delta)^{\frac{1}{2}}(u^\delta-v_\delta)\\
 \nonumber \le &  u^\delta_t + H_\ell(t,x, u^\delta, \nabla u^\delta)+\eps (-\Delta )^{\frac{1}{2}}u^\delta - (v_\delta)_t -H(t,x, v_\delta, \nabla v_\delta)-\eps (-\Delta )^{\frac{1}{2}}v_\delta\\
\nonumber   \le & 2\vartheta,
  \end{align} which clearly implies that
   \begin{align*}
    \varphi_t(t,x) - A|\nabla \varphi(t,x)|- B|\ell| + \lambda (u^\delta(t,x)-v_\delta(t,x)) + \epsilon(-\Delta)^{\frac{1}{2}}\varphi(t,x) \le 2\vartheta.
   \end{align*} This establishes \eqref{eq:eq_for_perturbed_fd}, and thereby proves  \eqref{eq:conclu_perturb_1}.  The proof of \eqref{eq:conclu_perturb_2} is  similar.

     \end{proof}

   \section{The law of diminishing oscillation and $C^{1,\alpha}$ estimate}
     \subsection{The law of diminishing oscillation}
  In our quest to prove $C^{1,\alpha}$-type regularity, the next proposition plays a pivotal role.
\begin{prop}\label{thm:diminshing_oscillation}
Let $u$ be an upper semicontinuous function such that $u\le 1$ in $[-2,0]\times \R^n$ .  Also, $u$ satisfies
\begin{align}
\label{eq:osc_sub_sol}
u_t - A |\nabla u| +  \epsilon(-\Delta)^{\frac{1}{2}} u \le \vartheta_0,
\end{align} interpreted in the viscosity sense, in $[-2,0]\times B_{2+2A}$. Assume further that there is a $\mu > 0$ such that
\begin{align*}
|\{u \le 0\}\cap [-2,-1]\times B_1| \ge \mu.
\end{align*}
 Then, for sufficiently small $\vartheta_0$, there is a $\theta \in (0,1))$  such that $ u\le 1-\theta $ in $[-1,0]\times B_1$. (The maximal value of $\theta$ depends on $A,~ \epsilon$ and $n$.)
\end{prop}
\begin{rem}
 The proof of Proposition \ref{thm:diminshing_oscillation} is a straightforward adaptation of similar results by Silvestre \cite{silvestre:2010}. However, for the sake of completeness of our presentation, we provide the full details.
\end{rem}

\begin{proof}[Proof of Proposition \ref{thm:diminshing_oscillation}]

 Consider the following initial value ODE
 \begin{align}
  \label{eq:first_ode}\begin{cases}\frac{d\gamma(t)}{dt} &= c_0 |x\in B_1: u(t,x) \le 0| - C_1\gamma(t)\quad \text{for}\quad t > -2, \\
   \gamma(-2) &= 0,
   \end{cases}
 \end{align} where $\gamma:[-2,0]\mapsto\R $ be a real valued function. The solution of the initial value problem \eqref{eq:first_ode} could be written explicitly, which is

 \begin{align*}
 \gamma(t) = \int_{-2}^t c_0|x\in B_1: u(s,x) \le 0|e^{-C_1(t-s)} ds.
 \end{align*} The goal is to show that, for a possibly large positive constant  $C_1$ and a possibly small positive constant  $c_0$, one has

 \begin{align}
 \label{eq:osci-1} u \le 1-\gamma(t) + 2\vartheta_0.
 \end{align}
By Fubini's theorem, for $t\in [-1,0]$, we have
\begin{align*}
\gamma(t) \ge c_0 e^{-2C_1} |\{u \le 0\}\cap [-2,-1]\times B_1| \ge  c_0 e^{-2C_1} \mu.
\end{align*} One can set $\theta = c_0 e^{-2C_1}\frac{\mu}{2}$ along with the choice that $\vartheta_0 \le  c_0 e^{-2C_1}\frac{\mu}{4}$ and the proposition follows from \eqref{eq:osci-1}.

 Choose a smooth nonincreasing  function $\beta:\R\mapsto \R$ such that $\beta(x) = 1$ if $x\le 1$ and $\beta(x)= 0$ if $x \ge 2$.

Define $h(t,x) = \beta(|x|+ At)= \beta(|x|-A|t|)$ for $t\in [-2,0]$. The function $h(t,x)$ looks like a bump function when considered as a function of $x$ alone. At a point $x$ where $h= 0$ (i.e.  $|x| \ge 2-At$), $(-\Delta)^{\frac{1}{2}} h < 0$. Since $h$ is smooth, $(-\Delta)^{\frac{1}{2}} h$ is also continuous. Hence there is a constant $\delta_1 > 0$ such that

\begin{align*}
  (-\Delta)^{\frac{1}{2}}  h \le 0 \quad \text{if}\quad h<\delta_1.
\end{align*}

Assume, if possible, that $u(t,x) > 1-\gamma(t) + \vartheta_0(2+t)$ for some point $(t,x)\in [-1, 0] \times B_1$.  Also note that $h(t,x) = 1$ for all  $(t,x)\in [-1, 0] \times B_1$ i.e. $ 1-\gamma(t) + \vartheta_0(2+t) =  1-\gamma(t)h(t,x) + \vartheta_0(2+t) $ on $ [-1, 0] \times B_1$.   For small enough $\vartheta_0$, the aim is to arrive at a contradiction by looking at the maxima of the function
\begin{align}
\label{eq:test-function} w(t,x) = u(t,x) + \gamma(t) h(t,x) -\vartheta_0 (2+t).
\end{align} By our assumption, theres is a point $(t,x)\in [-1,0]\times B_1$ such that $w(t,x) > 1$. Let $(t_0, x_0)$ be the point where the function $w$ achieves its maximum i.e.

\begin{align*}
 w(t_0, x_0) = \max_{[-2,0]\times \R^n} w(t,x) .
\end{align*} This maximum is bigger than $1$, therefore it must be achieved inside the support of $h$.  Hence
$$ |x_0| < 2+ A|t_0| \le 2+ 2A.$$
In other words, the function $u$ satisfies \eqref{eq:osc_sub_sol} at $(t_0, x_0)$ in the viscosity sense. Define

\begin{align}
 \label{eq:test-function-final}\varphi(t,x) = w(t_0,x_0)-\gamma(t)h(t,x) + \vartheta_0(2+t).
\end{align}

Then $(u-\varphi)(t,x) = w(t_0,x_0)-w(t,x)$ and since $w$ achieves its global maximum at $(t_0,x_0)$; $u-\varphi$ has a global maximum at $(t_0, x_0)$ and $u(t_0, x_0) = \varphi(t_0,x_0)$. Therefore, for $\kappa > 0$ small enough, we must have
\begin{align}\label{eqn:oscillation1}
 \varphi_t(t_0,x_0) -A|\nabla \varphi(t_0,x_0)| + \mathcal{I}_\kappa^\epsilon(\varphi)(t_0,x_0) + \mathcal{I}^{\epsilon,\kappa}(u)(t_0,x_0) \le \vartheta_0.
\end{align} The next task is to estimate each term on the lefthand side.  To this end, notice that \newline$|\nabla \varphi(t_0,x_0)|= \gamma(t)|\nabla h(t_0,x_0)|$, and
\begin{align}
\nonumber\varphi_t(t_0, x_0) & = -\gamma^\prime(t_0)h(t_0, x_0) -\gamma(t_0) h_t(t_0,x_0) + \vartheta_0\\
\nonumber                                  & =  -\gamma^\prime(t_0)h(t_0, x_0) -\gamma(t_0) \beta^\prime(|x|-A|t|)A + \vartheta_0\\
     \label{eqn:oscillation3}                             & =  -\gamma^\prime(t_0)h(t_0, x_0) +\gamma(t_0) |\nabla h(t_0,x_0)| A + \vartheta_0
\end{align}
 The tricky part however is to obtain a refined estimate  for the term $\mathcal{I}_\kappa^\epsilon(\varphi)$. Let us choose $0< \kappa << 1$. At $t = t_0$, as a function of $x$ alone $u(t_0,\cdot)+ \gamma(t_0) h(t_0, \cdot)$ achieves its maximum at $x_0$. To this end, we denote $ \Omega = \{u(t_0,\cdot)\le 0\}\cap B_1$.
Furthermore,
 \begin{align}
  \notag&\frac{-1}{\epsilon C(n, 1)}\Big( \mathcal{I}_\kappa^\epsilon(\varphi)(t_0,x_0) + \mathcal{I}^{\epsilon,\kappa}(u)(t_0,x_0) \Big)\\
 \notag =  &\int_{B(0,\kappa)}\frac{\varphi(t_0, x_0+z)-\varphi(t_0,x_0)}{|z|^{n+1}} dz + \int_{B(0,\kappa)^C}\frac{u(t_0, x_0+z)-u(t_0,x_0)}{|z|^{n+1}} dz\\
\notag    =  &\int_{B(x_0,\kappa)}\frac{\varphi(t_0, y)-\varphi(t_0,x_0)}{|x_0-y|^{n+1}} dz + \int_{B(x_0,\kappa)^C}\frac{u(t_0, y)-u(t_0,x_0)}{|x_0-y|^{n+1}} dz\\
\notag    =  &\int_{B(x_0,\kappa)}\frac{\varphi(t_0, y)-\varphi(t_0,x_0)}{|x_0-y|^{n+1}} dy + \int_{B(x_0,\kappa)^C\cap \Omega}\frac{u(t_0, y)-u(t_0,x_0)}{|x_0-y|^{n+1}} dz \\&\notag\hspace{7cm}+  \int_{B(x_0,\kappa)^C\cap \Omega^C}\frac{u(t_0, y)-u(t_0,x_0)}{|x_0-y|^{n+1}} dy\\
\notag     \le &\int_{B(x_0,\kappa)}\frac{\varphi(t_0, y)-\varphi(t_0,x_0)}{|x_0-y|^{n+1}} dy + \int_{B(x_0,\kappa)^C\cap \Omega^C}\frac{\varphi(t_0, y)-\varphi(t_0,x_0)}{|x_0-y|^{n+1}} dy \\\notag&\hspace{7cm}+  \int_{B(x_0,\kappa)^C\cap \Omega}\frac{u(t_0, y)-u(t_0,x_0)}{|x_0-y|^{n+1}} dy\\
\notag     = & \int_{\R^n}\frac{\varphi(t_0, y)-\varphi(t_0,x_0)}{|x_0-y|^{n+1}} dy + \int_{B(x_0,\kappa)^C\cap \Omega}\frac{u(t_0, y)+\gamma(t_0)h(t_0,y)-u(t_0,x_0)-\gamma(t_0)h(t_0,x_0)}{|x_0-y|^{n+1}} dy\\
  \label{eq:osci-2}   = &-\gamma(t_0)\int_{\R^n}\frac{h(t_0, y)-h(t_0,x_0)}{|x_0-y|^{n+1}} dy + \int_{B(x_0,\kappa)^C\cap \Omega}\frac{u(t_0, y)+\gamma(t_0)h(t_0,y)-u(t_0,x_0)-\gamma(t_0)h(t_0,x_0)}{|x_0-y|^{n+1}} dy.
      \end{align}

 Let $z\in \R^n$ such that $u(t_0, z)\le 0$, then
 \begin{align}
 \label{eq:osci-3}u(t_0,z)+ \gamma(t_0)h(t_0,z)-u(t_0,x_0)-\gamma(t_0) h(t_0,x_0) \le \gamma(t_0)-1,
\end{align}
 as $ u(t_0,x_0)+\gamma(t_0)h(t_0,x_0) = w(t_0,x_0)+\vartheta_0(2+ t_0) > 1$.

 Choose $c_0$ small enough such that  $\gamma(t_0) < \frac{1}{2}$, then by \eqref{eq:osci-3}

\begin{align}
\label{eq:osci-4}u(t_0,z)+ \gamma(t_0)h(t_0,z)-u(t_0,x_0)-\gamma(t_0) h(t_0,x_0) \le -\frac{1}{2}.
\end{align}

Therefore, for  such a choice of $c_0$, we must have by \eqref{eq:osci-2} and \eqref{eq:osci-4}
\begin{align}
\notag&\frac{-1}{\epsilon C(n, 1)}\Big( \mathcal{I}_\kappa^\epsilon(\varphi)(t_0,x_0) + \mathcal{I}^{\epsilon,\kappa}(u)(t_0,x_0) \Big)\\
\notag\le &  -\gamma(t_0)\int_{\R^n}\frac{h(t_0, y)-h(t_0,x_0)}{|x_0-y|^{n+1}} dy -\frac{1}{2}  \int_{B(x_0,\kappa)^C\cap \Omega} \frac{1}{|z-x_0|^{n+1} }\, dz\\
\label{eq:osci-5}\le& \frac{\gamma(t_0)}{C(n,1)} (-\Delta)^{\frac{1}{2}} h(t_0,x_0) -C_0 |\Omega -B(x_0,\kappa)|
\end{align} where  $C_0$ is a universal constant.

Note that as $\kappa \rightarrow 0$, the measure of the set $|\Omega \backslash B(x_0,\kappa)| \rightarrow |\Omega|$.  We now have to consider two different scenarios depending on $h(t_0,x_0)$ and arrive at contradictions in both cases.  We definitely have either $h(t_0, x_0) \le \delta_1$ or $h(t_0,x_0) >\delta_1$. 

 In the case where $h(t_0,x_0) \le \delta_1$,  one has $ (-\Delta )^{\frac{1}{2}} h(t_0,x_0) \le 0$. Hence

\begin{align}\label{eqn:oscillation-4}
 \Big(\mathcal{I}_\kappa^\epsilon(\varphi)(t_0,x_0) + \mathcal{I}^{\epsilon,\kappa}(u)(t_0,x_0) \Big) \ge \epsilon C(n,1)C_0|\Omega \backslash B(x_0,\kappa)|
\end{align} We plug \eqref{eqn:oscillation-4} and \eqref{eqn:oscillation3} into \eqref{eqn:oscillation1} to obtain

\begin{align*}
 -\gamma^\prime(t_0) h(t_0,x_0) + \gamma(t_0) |\nabla h(t_0,x_0)| A -\gamma(t_0) |\nabla h (t_0,x_0)| A +\mathcal{I}_\kappa^\epsilon(\varphi)(t_0,x_0) + \mathcal{I}^{\epsilon,\kappa}(u)(t_0,x_0) +\vartheta_0\le \vartheta_0.
\end{align*} In other words

\begin{align*}
 \gamma^\prime(t_0) h(t_0,x_0) \ge   \epsilon C(n, 1)C_0 |\Omega \backslash B(x_0,\kappa)|,
\end{align*} which is a contradiction to \eqref{eq:first_ode} for any $C_1$ if $\kappa$ is small enough and $c_0$ is chosen small enough so that it satisfies  $c_0\le  \epsilon C(n,1)C_0$,.

 We now turn our attention to the other case $h(t_0,x_0)\ge \delta_1$.  Since $h(t_0,x_0)$ is a smooth function with compact support, we must have
 $|(-\Delta)^{\frac{1}{2}} h| \le C$ for some $C > 0$. Therefore we have the estimate

 \begin{align}
  \label{eqn:oscillation-5}  \mathcal{I}_\kappa^\epsilon(\varphi)(t_0,x_0) + \mathcal{I}^{\epsilon,\kappa}(u)(t_0,x_0)  \ge -\epsilon C\gamma(t_0)+c_0|\Omega \backslash B(x_0,\kappa)|.
 \end{align}  We plug \eqref{eqn:oscillation-5} and \eqref{eqn:oscillation3} into \eqref{eqn:oscillation1} and obtain

 \begin{align*}
- C^\prime \gamma(t_0)+c_0|\Omega \backslash B(x_0,\kappa)| -\gamma^\prime(t_0)h(t_0,x_0) \le 0.
 \end{align*}

 We replace $\gamma^\prime(t_0)$ by using \eqref{eq:first_ode} in above and pass to limit $\kappa\rightarrow 0$ to obtain

 \begin{align*}
 (C_1h(t_0,x_0) -C)\gamma(t_0)+ c_0(1-h(t_0,x_0))|\Omega| \le 0,
 \end{align*} which is contradiction under for large enough $C_1$ as $h(t_0,x_0) \ge \delta_1$.

\end{proof}

\begin{thm}[diminishing oscillation]\label{thm:diminishing_oscilation}
 Let $\xi, \zeta$ be two bounded continuous  functions satisfying the inequalities
 \begin{align*}
   \xi_t -A|\nabla \xi | +\epsilon (-\Delta)^{\frac{1}{2}} \xi \le 0\\
   \zeta_t + A|\nabla \zeta|+\epsilon(-\Delta)^{\frac{1}{2}} \zeta \ge 0
 \end{align*} in the viscosity sense in $Q_1 = [-1,0]\times B_1$. Furthermore,
 \begin{align}\label{eq:totality}
  \{(t,x)\in Q_1 :\xi \le 0\}\cup \{(t,x)\in Q_1: \zeta \ge 0\} = Q_1.
 \end{align} Then there are universal constants $\theta \in (0,1)$ and $\alpha_0 >0$ (depending only on $A$, $\epsilon$ and $n$) such that if
 \begin{align*}
 \max\big( |\xi|,|\zeta| \big) &\le 1\quad &\text{in} \quad Q_1 = [-1,0]\times B_1,\\
 \max\big(|\xi|, |\zeta|\big) &\le 2 |(4+4A)x|^\alpha -1 \quad &\text{in}\quad  [-1,0]\times B_1^c
 \end{align*} for some $0<\alpha< \alpha_0$, then
 \begin{align*}
  \min(\textrm{osc}_{Q_{1/(4+4A)}}\xi, \text{osc}_{Q_{1/(4+4A)}}\zeta) \le 2(1-\theta).
 \end{align*}

\end{thm}

     \begin{proof}
 Let $R = 4+ 4A$. Consider the following rescaled versions of $\xi$ and $\zeta$:
 \begin{align*}
  \tilde{\xi} = \xi\big(\frac{t}{R}, \frac{x}{R}\big)\quad\text{and}\quad \tilde{\zeta} = \zeta\big(\frac{t}{R}, \frac{x}{R}\big).
 \end{align*}

 By the condition \eqref{eq:totality}
 \begin{align*}
 \text{either}\quad\quad  |\{\tilde{\xi} \le 0\} \cap ([-2,-1]\times B_1)| \ge \frac{|B_1|}{2},\\
 \text{or}\quad\quad  |\{\tilde{\zeta} \ge 0\} \cap ([-2,-1]\times B_1)| \ge \frac{|B_1|}{2}.
 \end{align*} Without loss of generality we may assume that the former is  true and it would be enough if we prove that

 \begin{align}
\label{eq:revise-eq-1}\textrm{osc}_{Q_{1/R}}\xi \le 2(1-\theta).
 \end{align} 
 
 To this end, if we were able to apply Proposition \ref{thm:diminshing_oscillation} to $\tilde{\xi}$, then there would exist a constant $\tilde{\theta}\in (0,1)$, depending only on $A,~\eps$ and $n$, such that
   \begin{align*}
        \sup_{(t,x)\in Q_1}\tilde{\xi}(t,x) \le 1-\tilde{\theta}\quad \text{i.e.} \quad \sup_{(t,x)\in Q_{1/R}}\xi(t,x) \le (1-\tilde{\theta}),
   \end{align*} which implies 
   \begin{align}
  \label{eq:revise-eq-2} \textrm{osc}_{Q_{1/R}}\xi \le (2-\tilde{\theta}).
   \end{align}Clearly, \eqref{eq:revise-eq-1} would then follow from \eqref{eq:revise-eq-2} if we simply choose $\theta =\frac{\tilde{\theta}}{2}$.

 Only condition that is missing here is that $ \tilde{\xi}$ needed to be bounded above by $1$.  To this end we define
 \begin{align*}
  u = \min(1,\tilde{\xi})
 \end{align*} and identify the inequality satisfied by $u$ in the viscosity sense.   
\vspace{.2cm}

  Note that inside $Q_R$, $\tilde{\xi} \le 1$ i.e. $ u= \tilde{\xi}$ in $Q_R$. Let $\varphi$ be a test function such that $u-\varphi$ has a global  maximum at $(t,x)$ and $u(t,x)=\varphi(t,x)$ where $(t,x) \in Q_{2+2A}$. Therefore $\tilde{\xi}-\varphi$ also has  a maximum at $(t,x)\in Q_{2+2A}$, which is global in $Q_R$. Notice that the point $(t,x)$ may not be a point of global maximum  in $[-(2+2A), 0]\times \R^n$ for the function  $\tilde{\xi}-\varphi$.  However, also note that $[-(2+2A), 0]\times B_\kappa(x) \subset [-(2+2A), 0]\times B_1(x)\subset Q_R$ for all $\kappa \in (0,1)$. Therefore, we can modify $\varphi$ outside $ [-(2+2A), 0]\times B_1(x)$ and obtain another test function $\tilde{\varphi}$ such that 
  \begin{align}
 \label{eq:revise-eq-3} \tilde{\varphi}(s,y) = \varphi(s,y)\quad\text{if}\quad (s,y) \in [-(2+2A), 0]\times B_1(x)
   \end{align}
  and $\tilde{\xi}-\tilde{\varphi}$ has a global maximum at $(t,x)$ in $[-R,0]\times\R^n$.

  Therefore at $(t,x)$, for all $\kappa\in (0,1)$,

  \begin{align}
 \label{eq:revise-eq-4}  \partial_t \tilde{\varphi}(t,x) -A|\nabla\tilde{ \varphi} (t,x)| +\mathcal{I}^\epsilon_\kappa \tilde{\varphi}(t,x) +\mathcal{I}^{\epsilon,\kappa} \tilde{\xi}(t,x) \le 0.
   \end{align} At this point we invoke \eqref{eq:revise-eq-3}, and then \eqref{eq:revise-eq-4} simply becomes 
   
    \begin{align}
 \label{eq:revise-eq-5}  \partial_t \varphi(t,x) -A|\nabla \varphi (t,x)| +\mathcal{I}^\epsilon_\kappa \varphi(t,x) +\mathcal{I}^{\epsilon,\kappa} \tilde{\xi}(t,x) \le 0.
   \end{align}

    We now estimate the quantity  $\mathcal{I}^{\epsilon,\kappa}u(t,x)$ and proceed as follows. For all $\kappa \in (0,1)$,

\begin{align*}
 &\mathcal{I}^{\epsilon,\kappa}\tilde{\xi}(t,x)-\mathcal{I}^{\epsilon,\kappa}u(t,x)\\
\le & \epsilon C(n, 1) \int_{x+y\notin B_{4+4A}} \big(\tilde{\xi}(t,x+y) -1\big)^{+} \frac{dy}{|y|^{n+1}}\\
\le &  \epsilon C(n, 1) \int_{x+y\notin B_{4+4A}} 2\big(|x+y|^{\alpha_0} -1\big)\frac{dy}{|y|^{n+1}} \\
\le & \vartheta_0, \quad(\vartheta_0~ \text{is from  Proposition \ref{thm:diminshing_oscillation}})
\end{align*} which holds, as a result of Fatou's lemma, for small enough $\alpha_0$. Therefore, the function $u$ satisfies the inequality 
\begin{align*}
   u_t -A|\nabla u|+\epsilon (-\Delta)^{\frac{1}{2}} u \le \vartheta_0
\end{align*} in the viscosity sense, for a suitably chosen $\vartheta_0$, in $[-2,0]\times B_{2+2A}$. Therefore we can apply  Proposition \ref{thm:diminshing_oscillation}  to $u$ and conclude the theorem.
\end{proof}

    \subsection{ $C^{1,\alpha}$-regularity: the end game}
 We begin this subsection with the following lemma.
\begin{lem} \label{lem:scalinglemma}
  Let $u$ be a bounded continuous function in $[0,T]\times \R^n$  such that it satisfies the inequalities
  \begin{align}
   \label{eq:revise-eq-6} u_t -A|\nabla u| -B +\lambda u+\epsilon(-\Delta)^{\frac{1}{2}} u \le 0, \\
   \label{eq:revise-eq-7} u_t +A|\nabla u| +B +\lambda u+\epsilon(-\Delta)^{\frac{1}{2}} u \ge 0
  \end{align} in the viscosity sense in $[0,T]\times \R^n$. Define
  \begin{align*}
  \xi(t,x) = \begin{cases}e^{\lambda t} u(t,x) -\frac{B}{\lambda}\big(e^{\lambda t} -1\big)\quad\text{if}\quad \lambda \neq 0\\
                       u(t,x) -Bt \quad \text{if }\quad \lambda = 0.
        \end{cases}
  \end{align*}

  \begin{align*}
  \zeta(t,x) = \begin{cases}e^{\lambda t} u(t,x) +\frac{B}{\lambda}\big(e^{\lambda t} -1\big)\quad\text{if}\quad \lambda \neq 0\\
                       u(t,x) +Bt \quad \text{if }\quad \lambda = 0.
        \end{cases}
  \end{align*}  Then $\xi$ and $\zeta$ are two bounded continuous  functions on $[0,T]\times \R^n$ and satisfy the inequalities
 \begin{align}
   \label{eq:transformed-1}\xi_t -A|\nabla \xi | +\epsilon (-\Delta)^{\frac{1}{2}} \xi \le 0\\
   \label{eq:transformed-2} \zeta_t + A|\nabla \zeta|+\epsilon(-\Delta)^{\frac{1}{2}} \zeta \ge 0
 \end{align} in the viscosity sense in $[0,T]\times \R^n$.
 \end{lem}


\begin{proof}
 Let $(t_0,x_0)\in (0,T)\times \R^n$ be a point and $\varphi$ be a test function such that $\varphi(t_0,x_0) = \xi(t_0,x_0)$ and  $\xi-\varphi$ has a global maximum at $(t_0,x_0)$.  Now define
   \begin{align*}
  \psi(t,x) = \begin{cases}e^{-\lambda t} \varphi(t,x) +\frac{B}{\lambda}\big(1-e^{-\lambda t}\big)\quad\text{if}\quad \lambda \neq 0\\
                       \varphi(t,x) +Bt \quad \text{if }\quad \lambda = 0.
        \end{cases}
  \end{align*} Then $u(t,x)-\psi(t,x) = e^{-\lambda t}\big(\xi(t,x)- \varphi(t,x)\big)$, which means $(t_0,x_0)$ is also a point of global maximum for $u-\psi$. Hence,

  \begin{align*}
  &\psi_t(t_0,x_0) -A|\nabla \psi(t_0,x_0)| -B +\lambda \psi(t_0,x_0)+\epsilon(-\Delta)^{\frac{1}{2}} \psi(t_0,x_0) \le 0\\
  \text{i.e.} \quad\quad & e^{-\lambda t_0}\Big[  \varphi_t -A|\nabla \varphi(t_0,x_0) | +\epsilon (-\Delta)^{\frac{1}{2}} \varphi(t_0,x_0)\Big] \le 0\\
   \text{i.e.} \quad\quad &   \varphi_t -A|\nabla \varphi(t_0,x_0) | +\epsilon (-\Delta)^{\frac{1}{2}} \varphi(t_0,x_0)\le 0,
  \end{align*}   which proves that \eqref{eq:transformed-1} is satisfied in the viscosity sense. The proof that  \eqref{eq:transformed-2} holds in the viscosity sense is similar.

\end{proof}

\begin{rem}
    Let $v$ be the unique Lipschitz continuous viscosity solution of \eqref{eq:HJB-eq-viscous-1}, and $l$ be a unit vector in $\R^n$.  For a nonzero scalar $h$, define $u(t,x) = \frac{v(t, x+hl)-v(t,x)}{|h|}$, a difference quotient of $v$ along $l$. Clearly, $v(\cdot, \cdot+ hl)$ is the unique Lipschitz continuous solution of \eqref{eq:HJB-eq-viscous-pert} with initial condition $ v(0, \cdot+hl)$ and $\ell = hl$.  Therefore, we can apply Lemma \ref{lemma: equation for the finte difference} and conclude that $v(\cdot, \cdot+hl)-v(\cdot,\cdot)$ satisfies \eqref{eq:conclu_perturb_1} and \eqref{eq:conclu_perturb_2} with $\ell = h l$. Furthermore, we invoke the Lipschitz continuity of $v$ in space to conclude that $u(t,x)$ is bounded and continuous, and divide the inequalities \eqref{eq:conclu_perturb_1} and \eqref{eq:conclu_perturb_2} throughout by $|h|$ ($=|\ell|$) to see that $u(t,x)$ satisfies the inequalities \eqref{eq:revise-eq-6} and \eqref{eq:revise-eq-7} in the viscosity sense. The next theorem is about establishing H\"{older} continuity estimates for a bounded and continuous  function, such as $u(t,x)$, that satisfies  \eqref{eq:revise-eq-6} and \eqref{eq:revise-eq-7}.
\end{rem}

\begin{thm}($C^{0,\alpha}$-estimate)\label{thm:linearized}
 Let $u$ be a bounded continuous function in $[0,T]\times \R^n$  such that it satisfies the inequalities
  \begin{align*}
   u_t -A|\nabla u| -B +\lambda u+\epsilon(-\Delta)^{\frac{1}{2}} u \le 0, \\
   u_t +A|\nabla u| +B +\lambda u+\epsilon(-\Delta)^{\frac{1}{2}} u \ge 0
  \end{align*} in the viscosity sense in $(0,T)\times \R^n$.  Then there is a positive constant $\alpha > 0$ (depending on $A, B, n$ and $\epsilon$) such that, for every $t > 0$ and $x\in \R^n$, the function $u$ is H\"{o}lder continuous of exponent $\alpha$ at $(t,x)$. Moreover, it follows that there exist constants $C$ and $K$ (depending on $A, n, \epsilon$ ) such that 
  \begin{align}
 \label{holder-regularity} |u(t,x)-u(s,y)| \le C(||u(0,\cdot)||_{L^\infty} + K)\Big[\frac{|x-y|^\alpha+|t-s|^\alpha}{t^\alpha}\Big]
  \end{align} for all $x, y\in \R^n$ and $0< s\le t< T$.

\end{thm}
     
     \begin{proof}
We fix $(t_0,x_0)\in (0, T)\times \R^n$ and consider the following transformations of $u(t,x)$:

\begin{align}
\label{eq:transformation-1} \xi(t,x) = \begin{cases}\frac{e^{\lambda t_0(t+1)}}{e^{\lambda T}\big(||u||_{L^\infty(Q_T)} + 2C_0(B,T,\lambda)+ C_1(A,B,n,\epsilon)\big)} \big[u(t_0(t+1),x_0+ t_0 x) -\frac{B}{\lambda}\big(1-e^{-\lambda t_0(t+1)} \big)\big]\quad\text{if}\quad \lambda \neq 0\\
                     \frac{1}{\big(||u||_{L^\infty(Q_T)} + 2C_0(B,T,0) + C_1(A,B,n,\epsilon)\big)} \big[  u(t_0(t+1),x_0+ t_0x) -Bt_0(t+1) \big] \quad \text{if }\quad \lambda = 0;
        \end{cases}
\end{align} and

\begin{align}
\label{eq:transformation-2} \zeta(t,x) = \begin{cases}\frac{e^{\lambda t_0(t+1)}}{e^{\lambda T}\big(||u||_{L^\infty(Q_T)} + 2C_0(B,T,\lambda)+ C_1(A,B,n)\big)} \big[u(t_0(t+1),x_0+t_0x) +\frac{B}{\lambda}\big(1-e^{-\lambda t_0(t+1)} \big)\big]\quad\text{if}\quad \lambda \neq 0\\
                     \frac{1}{\big(||u||_{L^\infty(Q_T)} + 2C_0(B,T, 0)+ C_1(A,B,n,\epsilon)\big)} \big[  u(t_0(t+1),x_0+t_0x) +Bt_0(t+1) \big] \quad \text{if }\quad \lambda = 0,
        \end{cases}
\end{align} where $(t,x)\in [-1,0]\times \R^n$ and $C_1(A,B,n,\epsilon)$ is a positive constant, depending on the quantities in the parenthesis, to be chosen later and $C_0(B,T,\lambda)$ has the following form
\begin{align*}
C_0(B,T,\lambda) =\begin{cases} \frac{B}{\lambda}\big(1-e^{-\lambda T}\big)\quad \text{if}\quad \lambda \neq 0\\
                             BT \quad \text{if}\quad\lambda = 0.
\end{cases}
\end{align*} A $C^{0,\alpha}$-type estimate for $\xi$ or $\zeta$ at the point $(0,0)$ would result in a $C^{0,\alpha}$ estimate for $u$ at $(t_0,x_0)$. To this end, we define
\begin{align}
\label{eq:transformation-3} \eta(t,x) = \begin{cases}\frac{e^{\lambda t_0(t+1)}}{e^{\lambda T}\big(||u||_{L^\infty(Q_T)} + 2C_0(B,T,\lambda)+ C_1(A,B,n,\epsilon)\big)}u(t_0(t+1),x_0+t_0 x)\quad\text{if}\quad \lambda \neq 0\\
                     \frac{1}{\big(||u||_{L^\infty(Q_T)} + 2BT + C_1(A,B,n,\epsilon)\big)}   u(t_0(t+1),x_0+t_0x)  \quad \text{if }\quad \lambda = 0.
        \end{cases}
\end{align} for $(t,x)\in [-1,0]\times \R^n$  and 
    \begin{align}
           \label{eq:transformation-4} F_\lambda(t)= \begin{cases}\frac{1}{e^{\lambda T}\big(||u||_{L^\infty(Q_T)} + 2C_0(B,T,\lambda)+ C_1(A,B,n,\epsilon)\big)}\frac{B}{\lambda}\big(e^{\lambda t_0(t+1)} -1\big)\quad\text{if}\quad \lambda \neq 0\\
                     \frac{1}{\big(||u||_{L^\infty(Q_T)} + 2C_0(B,T,\lambda) + C_1(A,B,n,\epsilon)\big)} Bt_0(t+1)\quad \text{if }\quad \lambda = 0
        \end{cases}
\end{align} for  $-1\le t \le 0$. Then clearly, $F_\lambda(t)$ is a smooth function of $t$ and is nonnegative if $t\in [-1,0]$. Also, for $\lambda > 0$, 
\begin{align}
 \sup_{-1\le t\le 0}|F_\lambda^\prime(t)| = &\frac{B t_0 e^{\lambda t_0}}{e^{\lambda T}\big(||u||_{L^\infty(Q_T)} + 2C_0(B,T,\lambda) + C_1(A,B,n,\epsilon)\big)}\nonumber\\
                  \label{eq:revision-new-1}    \le & \frac{B T }{\big(||u||_{L^\infty(Q_T)} + 2C_0(B,T,\lambda) + C_1(A,B,n,\epsilon)\big)}.
\end{align} The estimate \eqref{eq:revision-new-1} holds for $\lambda = 0$ as well.

 Furthermore, we rewrite \eqref{eq:transformation-1} and \eqref{eq:transformation-2} as 
\begin{align}
\label{eq:transformation-1.2} \xi(t,x) = \eta(t,x)-F_\lambda(t)\\
\label{eq:transformation-2.2} \zeta(t,x) = \eta(t,x)+F_\lambda(t)
\end{align} where $(t,x)\in [-1,0]\times \R^n$.

Then by Lemma \ref{lem:scalinglemma}, the functions $\xi$ and $\zeta$ are two bounded continuous  functions satisfying the inequalities
 \begin{align}
   \label{eq:transformed-1-2}\xi_t -A|\nabla \xi | +\epsilon (-\Delta)^{\frac{1}{2}} \xi \le 0\\
   \label{eq:transformed-2-2} \zeta_t + A|\nabla \zeta|+\epsilon(-\Delta)^{\frac{1}{2}} \zeta \ge 0
 \end{align} in the viscosity sense in $[-1,0]\times \R^n$. Also note that
 \begin{align}
\label{eq:holder-new-1} \max\{|\xi(t,x)|,|\zeta(t,x)|,|\eta(t,x)|\} \le 1\quad \text{for all}\quad (t,x)\in [-1,0]\times \R^n.
 \end{align} Set $r = \frac{1}{4+4A}$. We wish to show that there exist $\alpha > 0$ such that
 \begin{align}\label{eq:osc-holder}
 \max\Big(\text{osc}_{Q_{r^k}}\xi(t,x), \text{osc}_{Q_{r^k}}\zeta(t,x),\text{osc}_{Q_{r^k}}\eta(t,x)\Big)\le 2 r^{\alpha k}\quad \quad \text{for all}\quad\quad k= 0,1,2,.....
 \end{align} The estimate \eqref{eq:osc-holder} is a necessary and sufficient  for H\"{o}lder continuity (with exponent $\alpha$) of $\xi, \zeta$ and $\eta$ at $(0,0)$.  This H\"{o}lder continuity of $\eta$ at $(0,0)$ is equivalent to H\"{o}lder continuity of $u$ at $(t_0,x_0)$. To prove \eqref{eq:osc-holder}, we follow \cite{silvestre:2010} and find a sequence  $(a_k, b_k)$, $k=0,1,2,........$  such that

\begin{align*}
 a_k\le \xi(t,x),  \eta(t,x),  \zeta(t,x) \le b_k\quad \text{for all }\quad (t,x)\in Q_{r^k},
\end{align*} with $b_k-a_k = 2r^{\alpha k}$ where  $\{ a_k\}$ is nondecreasing and $\{b_k\}$ is nonincreasing. We employ the method of induction.  
We invoke \eqref{eq:holder-new-1} and  choose $a_0\le \min\big(\inf_{Q_1}\xi(t,x),\inf_{Q_1}\zeta(t,x),\inf_{Q_1}\eta(t,x)\big)$ and $b_0\ge \max\big(\sup_{Q_1}\xi(t,x),\sup_{Q_1}\zeta(t,x),\sup_{Q_1}\eta(t,x)\big)$ such that $b_0-a_0 = 2$. We now assume that the sequence $(a_m, b_m)$ has been constructed up to some index $k$. We want to show the existence of $(a_{k+1}, b_{k+1})$.To this end, define
\begin{align*}
 \xi_k(t,x)= \big(\xi(r^kt,r^kx)-\frac{a_k+b_k}{2}\big)r^{-\alpha k} ,\\
 \zeta_k(t,x)= \big(\zeta(r^kt,r^kx)-\frac{a_k+b_k}{2}\big)r^{-\alpha k},\\
  \eta_k(t,x)= \big(\eta(r^kt,r^kx)-\frac{a_k+b_k}{2}\big)r^{-\alpha k}.
\end{align*} In view of \eqref{eq:transformation-1.2} and \eqref{eq:transformation-2.2}, 
\begin{align}
\label{eq:revision-new-5} \xi_k(t,x) = \eta_{k}(t,x) -r^{-\alpha k}F_\lambda(r^k t)\quad \text{and}\quad  \zeta_k(t,x) = \eta_{k}(t,x) +r^{-\alpha k}F_\lambda(r^k t).
\end{align}

\vspace{.2cm}
\noindent{\bf Claim 1}: $Q_1 = \big\{(s,y)\in Q_1: \xi_k(s,y)\le 0\big\}\bigcup  \big\{(s,y)\in Q_1: \zeta_k(s,y)\ge  0\big\}$. 

\noindent{\it Justification}: Note that $r^{-\alpha k} F_\lambda(r^kt)$ is nonnegative, and for given $(t,x)\in Q_1$ either $\eta_k(t,x) \le r^{-\alpha k} F_\lambda(r^kt)  $ or $\eta_k(t,x) \ge r^{-\alpha k} F_\lambda(r^kt)  $.  If $\eta_k(t,x) \le r^{-\alpha k} F_\lambda(r^kt)  $, then $(t,x)\in \big\{(s,y)\in Q_1: \xi_k(s,y)\le 0\big\}$. Otherwise, since $r^{-\alpha k} F_\lambda(r^kt)$ is nonnegative, $\eta_k(t,x) \ge r^{-\alpha k} F_\lambda(r^kt) $ implies $\eta_k(t,x) \ge -r^{-\alpha k} F_\lambda(r^kt) $ i.e. $ (t,x)\in\big\{(s,y)\in Q_1: \zeta_k(s,y)\ge  0\big\}$, and the claim follows.

\vspace{.2cm}
\noindent{\bf Claim 2}: It holds that 
\begin{align*}
 &\max\{|\xi_k(t,x)|, |\zeta_k(t,x)|\} \le 1 \quad \text{in}\quad Q_1,
\end{align*} and 
\begin{align}
\label{eq:holder_oscilllation}\max\{|\xi_k(t,x)|, |\zeta_k(t,x)|\} \le 2|r^{-1}x|^{\alpha }-1 \quad\quad\text{if}\quad\quad |x| >1.
\end{align}

\noindent{\it Justification}: The first part of the claim is a simple consequence of the induction hypothesis on $(a_k, b_k)$. For the second part we argue as follows: let $m\in \{1,2........, k\}$ be an integer and $(t,x)\in Q_{r^{-m}}$. Then $(r^k t, r^k x)\in Q_{r^{k-m}}$, and 
\begin{align*}
  \xi_k(t,x) &= \big(\xi(r^k t, r^k x)-\frac{b_k-a_k}{2} -a_k\big)r^{-\alpha k}\\
                  & =  \big(\xi(r^k t, r^k x) -a_k\big)r^{-\alpha k} -1\\
                  & \le  \big(\xi(r^k t, r^k x) -a_{k-m}\big)r^{-\alpha k} -1\\
                  &\le 2 r^{\alpha(k-m)} r^{-\alpha k} -1 = 2r^{-\alpha m}-1.
\end{align*} Also,
   
\begin{align*}
 - \xi_k(t,x) &= -\big(\xi(r^k t, r^k x)-\frac{a_k-b_k}{2} -b_k\big)r^{-\alpha k}\\
                  & =  \big(b_k-\xi(r^k t, r^k x)\big)r^{-\alpha k} -1\\
                  & \le  \big(b_{k-m}-\xi(r^k t, r^k x) \big)r^{-\alpha k} -1\\
                  &\le 2 r^{\alpha(k-m)} r^{-\alpha k} -1 = 2r^{-\alpha m}-1.
\end{align*} Therefore 
\begin{align*}
   | \xi_k(t,x)|\le 2 r^{-\alpha m} -1\quad \text{if}\quad (t,x)\in Q_{r^{-m}}, ~ m\in \{1,2........, k\}.
\end{align*} Similarly
\begin{align*}
   | \zeta_k(t,x)|\le 2 r^{-\alpha m} -1\quad \text{if}\quad (t,x)\in Q_{r^{-m}}, ~ m\in \{1,2........, k\}.
\end{align*} Combining, we obtain
 \begin{align} 
 \label{eq:revision-new-2.1}  \quad \quad& \max\{|\xi_k(t,x)|, |\zeta_k(t,x)|\} \le 2r^{-\alpha m}-1 \quad \text{in}\quad [-1,0]\times B_{r^{-m}}
\end{align} for $m =1,2,......,k$. Moreover, it also holds that 

 \begin{align} 
  \label{eq:revision-new-3.1}  \quad \quad& \max\{|\xi_k(t,x)|, |\zeta_k(t,x)|\} \le 2r^{-\alpha k}-1 \quad \text{in}\quad [-1,0]\times \R^n.
\end{align} Therefore, from \eqref{eq:revision-new-3.1} we conclude
\begin{align}
\label{eq:holder_oscilllation_1}\max\{|\xi_k(t,x)|, |\zeta_k(t,x)|\} \le 2|r^{-1}x|^{\alpha }-1 \quad\quad\text{if}\quad\quad |x| \ge r^{-(k-1)}.
\end{align} We now simply  combine \eqref{eq:revision-new-2.1} and \eqref{eq:holder_oscilllation_1} and complete the justification.
\vspace{.2cm}
Therefore, by Theorem \ref{thm:diminishing_oscilation}, there are universal constants $\alpha_0\in (0,1)$ and $\theta\in (0,1)$ (depending on $n, A, \epsilon$) such that if \eqref{eq:holder_oscilllation}
holds for any $0<\alpha < \alpha_0$, then
 \begin{align}
 \label{eq:revision-new-2} \text{either}\quad\quad \text{osc}_{Q_r} \xi_k \le 2(1-\theta)\quad\quad \text{or}\quad \quad \text{osc}_{Q_r} \zeta_k \le 2(1-\theta).
 \end{align}

\vspace{.2cm}
\noindent{\bf Claim 3}: As an implication of \eqref{eq:revision-new-2}, it holds that 
\begin{align}
 \label{eq:revision-new-3}\max\Big(\text{osc}_{Q_{r}}\xi_k, \text{osc}_{Q_{r}}\zeta_k, \text{osc}_{Q_{r }}\eta_k\Big)\le 2(1-\theta) + \frac{2r B T}{ \big(||u||_{L^\infty} + 2 C(B,T,\lambda) +C_1(A, B, n, \epsilon)\big)}.
\end{align}

\noindent{\it Justification}:  When the condition \eqref{eq:revision-new-2} is satisfied, without loss of generality we may assume that 
 \begin{align}
 \label{eq:revision-new-4} \quad\quad \text{osc}_{Q_r} \xi_k \le 2(1-\theta).
 \end{align} Then, by \eqref{eq:revision-new-5},
\begin{align*}
 \eta_{k}(t,x) =\xi_k(t,x) +r^{-\alpha k}F_\lambda(r^k t)\quad \text{and}\quad  \zeta_k(t,x) = \xi_{k}(t,x) +2 r^{-\alpha k}F_\lambda(r^k t).
\end{align*} Now,
               \begin{align}
              \notag  \text{osc}_{Q_{r}}\zeta_k = &\sup_{(t,x), (s,y)\in Q_r}|\zeta_k(t,x)-\zeta_k(s,y)|\\
               \notag                                             \le & \sup_{(t,x), (s,y)\in Q_r}|\xi_k(t,x)-\xi_k(s,y)| + \sup_{t,s \in [-r,0]} 2r^{-\alpha k}|F_\lambda(r^k t)-F_\lambda(r^k s)|\\
                   \notag                                         \le &   \text{osc}_{Q_{r}}\xi_k + 2r^{-\alpha k}\big( \sup_{\tau\in [-1,0] }|F_\lambda^\prime(\tau)|\big)\sup_{t,s \in [-r,0]} |r^k t-r^k s|\\
                    \notag                                        = &   \text{osc}_{Q_{r}}\xi_k+ 2\big( \sup_{\tau\in [-1,0] }|F_\lambda^\prime(\tau)|\big) r^{(1-\alpha)k} \sup_{t,s \in [-r,0]} |t-s |\\
                  \notag                                    \le & \text{osc}_{Q_{r}}\xi_k+ 2\big( \sup_{\tau\in [-1,0] }|F_\lambda^\prime(\tau)|\big) r\\
                     \label{eq:revision-new-6}           ( \text{By}~\eqref{eq:revision-new-1})\quad\quad                               \le & 2(1-\theta)+  \frac{2r B T}{ \big(||u||_{L^\infty} + 2 C(B,T,\lambda) +C_1(A, B, n, \epsilon)\big)}.
               \end{align}
Furthermore, 
 \begin{align}
              \notag  \text{osc}_{Q_{r}}\eta_k = &\sup_{(t,x), (s,y)\in Q_r}|\eta_k(t,x)-\eta_k(s,y)|\\
                \notag                                            \le & \sup_{(t,x), (s,y)\in Q_r}|\xi_k(t,x)-\xi_k(s,y)| + \sup_{t,s \in [-r,0]} r^{-\alpha k}|F_\lambda(r^k t)-F_\lambda(r^k s)|\\
                    \notag                                        \le &   \text{osc}_{Q_{r}}\xi_k +\big( \sup_{\tau\in [-1,0] }|F_\lambda^\prime(\tau)|\big)r^{(1-\alpha) k}\sup_{t,s \in [-r,0]} | t-s|\\
                    \label{eq:revision-new-7}      ( \text{By}~ \eqref{eq:revision-new-1})\quad\quad                               \le & 2(1-\theta)+  \frac{r B T}{ \big(||u||_{L^\infty} + 2 C(B,T,\lambda) +C_1(A, B, n, \epsilon)\big)}.
               \end{align} The claim follows simply by combining \eqref{eq:revision-new-7}, \eqref{eq:revision-new-6} and \eqref{eq:revision-new-4}.
  \vspace{.2cm}             
               
               Now choose $C_1(A, B, n, \epsilon)$ big enough such that
\begin{align*}
 \frac{2r B T}{ \big(||u||_{L^\infty} + 2 C(B,T,\lambda) +C_1(A, B, n, \epsilon)\big)} \le \theta.
\end{align*} Then, by \eqref{eq:revision-new-3}

\begin{align*}
\max\Big(\text{osc}_{Q_{r}}\xi_k(t,x), \text{osc}_{Q_{r}}\zeta_k(t,x),\text{osc}_{Q_{r }}\eta_k(t,x)\Big)\le 2(1-\frac{\theta}{2}).
\end{align*} Now choose $\alpha $, possibly strictly smaller than $\alpha_0$, such that
\begin{align*}
 \big(1-\frac{\theta}{2}\big) \le r^\alpha.
\end{align*} Therefore we finally have
\begin{align*}
\max\Big(\text{osc}_{Q_{r}}\xi_k(t,x), \text{osc}_{Q_{r}}\zeta_k(t,x),\text{osc}_{Q_{r }}\eta_k(t,x)\Big)\le 2 r^\alpha.
\end{align*} In other words

\begin{align}
\label{eq:eq:holder-new-2}\max\Big(\text{osc}_{Q_{r^{k+1}}}\xi(t,x), \text{osc}_{Q_{r^{k+1}}}\zeta(t,x),\text{osc}_{Q_{r^{k+1}}}\eta(t,x)\Big)\le 2 r^{\alpha(k+1)}.
\end{align}  Therefore the pair $(a_{k+1}, b_{k+1})$ could also be chosen. 

\end{proof}

      With the Theorem \ref{thm:linearized} and Lemma \ref{lemma: equation for the finte difference}  at  our disposal,  we can now follow the line of argument by Silvestre\cite{silvestre:2010} and prove Theorem \ref{thm:regularity_theorem}.

 \begin{proof}[Proof of Theorem \ref{thm:regularity_theorem}]
    Let  $l\in \R^n$ be a unit vector and $h$ be a nonzero constant. The difference quotient of $u^\epsilon$ along $l$ at $(t,x)$ is defined as

 \begin{align*}
 \partial_{h,l}u^\epsilon(t,x) = \frac{u^\epsilon(t,x+hl)-u^{\epsilon}(t,x)}{|h|}.
 \end{align*} For fixed $ h$ and $l$, the function $\partial_{h,l}u^\epsilon(t,x)$ is continuous and bounded by $||\nabla_x u^\epsilon(t,\cdot)||_{L^\infty}$, which is bounded above by a constant independent of $\epsilon$.  We now recall Lemma \ref{lemma: equation for the finte difference} and see that $u = \partial_{h,l}u^\epsilon(x) $ satisfies all the hypotheses of Theorem \ref{thm:linearized}. Therefore, there exists a positive constant $\alpha $ such that 
  \begin{align*}
   ||\partial_{h,l}u^\epsilon(t,x)||_{C^{0,\alpha}([\frac{t}{2},t]\times \R^n)} \le \frac{C(||\nabla u_0||_{L^\infty} + K)}{t^\alpha}
  \end{align*} uniformly in $h$ and $l$. We now apply Arzela-Ascoli's theorem and pass to the limit $h\rightarrow 0$ and conclude that $\partial_l u^\epsilon(t,x)$ exists and
   \begin{align*}
   ||\partial_{l}u^\epsilon(t,x)||_{C^{0,\alpha}([\frac{t}{2},t]\times \R^n)} \le \frac{C(||\nabla u_0||_{L^\infty} + K)}{t^\alpha}
  \end{align*} for all unit vectors $l\in \R^n $. In other words

    \begin{align*}
   ||\nabla u^\epsilon(t,x)||_{C^\alpha([\frac{t}{2},t]\times \R^n)} \le \frac{C(||\nabla u_0||_{L^\infty} + K)}{t^\alpha}.
  \end{align*} With this information at hand , we see that $u^\epsilon(t,x)$ is $C^{1,\alpha}$ in space for any fixed $t >0$. Therefore \newline$H(t,x,u, \nabla u) + \epsilon (-\Delta)^{\frac{1}{2}} u^\epsilon$ is bounded. Hence $\partial_t u^\epsilon$ is bounded for $t> 0$ i.e. the difference quotients in $t$
     \begin{align*}
     \partial_h u(t,x) = \frac{u(t+h,x)- u(t,x)}{h}
     \end{align*} is bounded independently of $h$. Also, it is easy see that $\partial_h u(t,x)$ satisfies all the hypothesis of Theorem \ref{thm:linearized} . Therefore $ \partial_h u(t,x)$ is bounded in $C^{0,\alpha}$ uniformly in $h$ and consequently $ \partial_t u(t,x)$ is $C^{0,\alpha}$ with the following estimate
      \begin{align*}
   ||\partial_{t}u^\epsilon||_{C^\alpha([\frac{t}{2},t]\times \R^n)} \le \frac{C(||\nabla u_0||_{L^\infty} + K)}{t^\alpha}.
  \end{align*}

 \end{proof}

   \section{Error estimate for vanishing viscosity approximation}
\begin{proof}[Proof of Theorem \ref{thm:vanishing viscosity rate}]
    The proof uses the  doubling of variables technique, wellknown in the viscosity solution theory.

     Let
     \begin{align}
     \label{eq:test-function-defi} \Phi(x,y) = \frac{\vartheta}{2} |x-y|^2 +\beta^2 |x|^2,
     \end{align} where $\vartheta$ and $\beta$ are two positive constants (to be chosen later), and define
          \begin{align*}
          \sigma_0 &=  \sup_{t\in [0,T], x, y\in\R^n}\Big\{u_0(x)-u_0(y)-\Phi(x,y)-\frac{\gamma}{T}\Big\}^+\\
       \sigma &= \sup_{t\in [0,T], x, y\in\R^n}\Big\{u(t,x)-u^\epsilon(t,y)-\Phi(x,y)-\frac{\gamma}{T-t}\Big\}-\sigma_0,
     \end{align*} where $\gamma \in (0,1)$. Next, we introduce the quantity

     \begin{align}
          \label{eq:indrouce-Psi} \Psi(t,x,y) = u(t,x)-u^\epsilon(t,y)- \Phi(x,y)-\frac{\delta\sigma}{T}t -\frac{\gamma}{T-t},
     \end{align} where $\delta \in (0,1)$. Recall that $u(t,x)$ and $u^\epsilon(t,y)$ are bounded and continuous functions. Therefore, thanks to
      the penalization term $\frac{\gamma}{T-t}$, there exists $(t_0,x_0,y_0)\in [0,T)\times \R^n\times \R^n$ such that
      \begin{align}
     \label{eq:re-revise} \Psi(t_0,x_0,y_0) = \sup_{t,x,y} \Psi(t,x,y).
      \end{align} We are interested in finding an upper bound on $\sigma+\sigma_0$ by deriving a positive upper bound on $\sigma$. Therefore, without loss of generality, we may assume that $\sigma > 0$. This implies $t_0 > 0$, since 
      
               \begin{align*}
      \Psi(t_0,x_0,y_0) \ge \sup_{t\in [0,T], x, y\in\R^n}\Big\{u(t,x)-u^\epsilon(t,y)-\Phi(x,y)-\frac{\gamma}{T-t}\Big\}- \delta\sigma = \sigma_0+ (1-\delta)\sigma > \sigma_0,
      \end{align*} while on the other hand $t_0 = 0$ would imply that $  \Psi(t_0,x_0,y_0) \le \sigma_0$.
      
       Therefore we can apply the maximum principle for semicontinuous functions, adapted to IPDEs\cite{Jakobsen:2005jy}, and conclude:

         For each $\kappa \in (0,1)$, there exists $a, b\in \R$ satisfying
         \begin{align*}
          a-b = \frac{\delta\sigma}{T}+\frac{\gamma}{(T-t_0)^2},
         \end{align*} such that

         \begin{align}
        \label{eq:semicont-1} &a+H(t_0,x_0,u(t_0,x_0) ,\nabla_x\Phi(x_0,y_0)) \le 0\\
        \label{eq:semicont-2} &b+H(t_0,y_0, u^\epsilon(t_0,y_0), -\nabla_y\Phi\big(x_0,y_0)\big)+ \mathcal{I}_{\kappa}^\eps\big(u^\epsilon(t_0, y_0)\big) +\mathcal{I}^{\eps,\kappa}\big(u^\epsilon(t_0, y_0)\big) \ge 0.
         \end{align}  We simply choose $\kappa=\epsilon$ in \eqref{eq:semicont-2}, and subtract it from \eqref{eq:semicont-1} to obtain

         \begin{align}
       \notag  \frac{\delta\sigma}{T}+\frac{\gamma}{(T-t_0)^2}+H(t_0,x_0,u(t_0,x_0), \nabla_x\Phi(x_0,y_0))-H(t_0,y_0, u^\epsilon(t_0,y_0), -\nabla_y\Phi\big(x_0,y_0)\big)\\
         \label{eq:semicont-3} - \mathcal{I}_{\epsilon}^\eps\big(u^\epsilon(t_0,y_0)\big) -\mathcal{I}^{\eps,\eps}\big(u^\epsilon(t_0,y_0)\big)\le 0.
         \end{align}

         We now denote $q_0=\vartheta(x_0-y_0)$ and  make the following claims:
         
          \vspace{.3cm} 
         \noindent{\bf Claim $1$:} It holds that
             \begin{align}
       \label{eq:re-revise-4}  |\mathcal{I}^{\eps,\eps}\big(u^\epsilon(t_0,y_0)\big)|\le C\epsilon(1+ |\log \epsilon|) \le  C\epsilon |\log \epsilon| \quad\text{if} \quad \epsilon \le e^{-1}.
          \end{align}
         
         \noindent{\bf Claim $2$:} For any $z\in \R^n$, it holds that  
         \begin{align}
                        \label{eq:re-revise-1}                    u_\epsilon(t_0,y_0+z)-u_\epsilon(t_0,y_0)-q_0.z \ge -\frac{\vartheta}{2}|z|^2.
         \end{align} For the time being we simply assume  \eqref{eq:re-revise-4}-\eqref{eq:re-revise-1} and proceed, a detailed justification will be provided at a later stage. 
         \vspace{.3cm}

          Now make the substitution $r=\frac{z}{\epsilon}$ and obtain

           \begin{align} \notag- \mathcal{I}_{\epsilon}^\eps\big(u^\epsilon(t_0,y_0)\big)= &\epsilon C(n,1) \int_{|z| \le \epsilon} \big[ u^\epsilon(t_0,y_0+z)-u^\epsilon(t_0,y_0) \big]\frac{1}{|z|^{n+1}} \,dz\\
      \notag =&\epsilon C(n,1) \int_{|z| \le \epsilon} \big[ u^\epsilon(t_0,y_0+z)-u^\epsilon(t_0,y_0)-q_0.z \big]\frac{1}{|z|^{n+1}} \,dz\\
         \notag  =&  C(n,1)  \int_{B(0,1)} \big[ u^\epsilon(t_0,y_0+\epsilon r)-u^\epsilon(t_0, y_0)-\epsilon q_0. r \big]\frac{1}{ |r|^{n+1}}\,dr\\
         \label{eq:rate_new _1}  \ge& - C \vartheta\epsilon^2,
          \end{align}where we have used \eqref{eq:re-revise-1}.

         Moreover, 
         $u(t_0,x_0) -u^\epsilon(t_0,y_0) \ge \frac{\delta \sigma t_0}{T}+ \frac{\gamma}{T-t_0} -\frac{\gamma}{T} \ge 0$ as $\sigma \ge 0$. Therefore, by the monotonicity and the Lipschitz continuity of $H$, we must have

     \begin{align}
     \label{eq:semicont-4} H(t_0,x_0,u(t_0,x_0),D_x\Phi(x_0,y_0))-H(t_0,y_0, u^\epsilon(t_0,y_0), -D_y\Phi\big(x_0,y_0)\big) \ge -K|x_0-y_0| -K\beta |x_0|.
     \end{align} Furthermore, $\nabla_y u^\epsilon(t_0,y_0) = \vartheta(y_0-x_0)$ i.e $\vartheta |x_0-y_0| \le ||\nabla u^\epsilon||_\infty \le C$. Also, it is easy to see that $\beta |x_0|^2 < C$ i.e $\beta |x_0| \le C\sqrt{\beta}$.  To this end, we combine \eqref{eq:semicont-3}-\eqref{eq:semicont-4} and obtain 
     
        \begin{align}
     \label{eq:semicont-6}  \frac{\delta\sigma}{T}\le -\frac{\gamma}{(T-t_0)^2} + \frac{K}{\vartheta} +K\sqrt{\beta} +C\epsilon|\log \epsilon|+C\vartheta\epsilon^2.
      \end{align} In addition, by Lipschitz continuity of $u_0(x)$ we have 
         \begin{align}
            \label{eq:semicont-7}   \sigma_0 \le \sup_{x,y\in\R^n} \big(K|x-y|-\frac{\vartheta}{2} |x-y|^2\big) = \frac{2K^2}{\vartheta}.
         \end{align} 
         Therefore
         \begin{align}
     \label{eq:semicont-8}  \frac{\delta(\sigma+\sigma_0)}{T}\le -\frac{\gamma}{(T-t_0)^2} + \frac{K^\prime}{\vartheta} +K^\prime\sqrt{\beta} +C\epsilon|\log \epsilon|+C\vartheta\epsilon^2.
      \end{align}
      
       Now choose $\beta = \frac{\gamma^2}{{K^{\prime}}^2T^2}$, and maximize the right hand side with respect to $\vartheta$ to obtain,

      \begin{align*}
       \frac{\delta(\sigma+\sigma_0)}{T} \le C\epsilon(1+|\log \eps|) \quad(\le C \eps|\log \eps|)\quad \text{if}\quad \eps \le e^{-1}.
      \end{align*} We now let $\delta \uparrow 1$, and conclude
      \begin{align*}
      u(t,x)-u^\epsilon(t,x)- \frac{\gamma^2}{K^2T^2}|x|^2-\frac{\gamma}{T-t} \le C \eps |\log \eps|.
      \end{align*} Finally, we let $\gamma \downarrow 0$ in the above inequality to conclude

 \begin{align}
     \label{eq:semicont-9} u(t,x)-u^\epsilon(t,x) \le C \eps |\log \eps|.
      \end{align} The inequality \eqref{eq:semicont-9} provides only half of requirement, the other half could also be concluded in a similar manner. Therefore the proof would be complete if we are able to justify {\bf Claim 1} and {\bf Claim 2}.

      \vspace{.2cm}
        \noindent{\it Justification of  Claim $1$:} 
          Let $K$ be the constant mentioned in \ref{A1}. Thanks to  \ref{A2}, the functions $-(Kt+||u_0||_{L^\infty})$ and $(Kt+||u_0||_{L^\infty})$ are respectively sub and super solutions of \eqref{eq:HJB-eq-viscous}. Hence by comparison principle 
          $ -(Kt+||u_0||_{L^\infty}) \le u^{\epsilon}(t,x) \le (Kt+||u_0||_{L^\infty})$ 
         i.e 
         \begin{align}
              \label{eq:global-bound}  ||u^{\epsilon}(t,x)||_{L^{\infty}([0,T]\times \R^n)} \le KT+ ||u_0||_{L^\infty}~\text{for all}~\epsilon\in [0,1].
         \end{align} 
        Moreover by Lemma \ref{lip}   there is  constant $L$ such that 
           \begin{align}
              \label{eq:global-bound-1}  ||\nabla_xu^{\epsilon}(t,\cdot)||_{L^{\infty}([0,T]\times \R^n)} \le L~\text{for all}~\epsilon\in [0,1].
         \end{align} Now 
         \begin{align*}
              &\Big|    \int_{|z|> \epsilon} \frac{u^\epsilon(t_0, y_0+z)-u^\epsilon(t_0, y_0)}{|z|^{n+1}}\,dz\Big| \\
              \le & \int_{ \epsilon\le |z|\le 1} \frac{|u^\epsilon(t_0, y_0+z)-u^\epsilon(t_0, y_0)|}{|z|^{n+1}}\,dz+\int_{  |z|\ge 1} \frac{|u^\epsilon(t_0, y_0+z)-u^\epsilon(t_0, y_0)|}{|z|^{n+1}}\,dz\\
              \le&   ||\nabla_xu^{\epsilon}(t,\cdot)||_{L^{\infty}([0,T]\times \R^n)}  \int_{ \epsilon\le |z|\le 1} \frac{\, dz}{|z|^{n}} +2  ||u^{\epsilon}(t,x)||_{L^{\infty}([0,T]\times \R^n)}  \int_{  |z|\ge 1} \frac{\, dz}{|z|^{n+1}}\\
                &(\text{By polar transform and \eqref{eq:global-bound}-\eqref{eq:global-bound-1}})\\
            \le & C \int_{\eps < r\le 1} \frac{\,dr}{r} + C^{\prime} \int_{ r> 1} \frac{\,dr}{r^2}\\
             \le&  C (|\log \eps| + 1). 
         \end{align*}  
      Therefore
            \begin{align*}
        & |\mathcal{I}^{\eps,\eps}\big(u^\epsilon(t_0,y_0)\big)|\notag  \\ &=  \epsilon C(n,1) \Big|    \int_{|z|> \epsilon} \frac{u^\epsilon(t_0, y_0+z)-u^\epsilon(t_0, y_0)}{|z|^{n+1}}\,dz\Big| \le C\epsilon(1+ |\log \epsilon|) \le  C\epsilon |\log \epsilon| \quad\text{if} \quad \epsilon \le e^{-1},
          \end{align*} which is what we wanted to show. 

\vspace{.2cm}      
      
 \noindent{\it Justification of  Claim $2$:} We recall \eqref{eq:re-revise} and note that $(t_0, x_0, y_0)$ is a point of global maximum of $\Psi(t,x,y)$. Hence for any $z\in\R^n$,
              \begin{align*}
               \Psi(t_0, x_0, y_0)& \ge \Psi(t_0,x_0, y_0+z)\\
               i.e \hspace{2cm} u(t_0, x_0)-u^\epsilon(t_0, y_0)-\Phi(x_0, y_0) 
               &\ge  u(t_0, x_0)-u^\epsilon(t_0, y_0+z)-\Phi(x_0, y_0+z) .\\
           i.e \hspace{2.5cm}    u^\epsilon(t_0, y_0+z)-u^\epsilon(t_0, y_0) \ge& -\Phi(x_0, y_0+z)+ \Phi(x_0, y_0)\\
                                 & = \frac{\vartheta}{2}|x_0-y_0|^2 -  \frac{\vartheta}{2}|(x_0-y_0)-z|^2\\
                                 & = \vartheta  (x_0-y_0). z\ -\frac{\vartheta}{2} |z|^2\\
                                 & = q_0. z-\frac{\vartheta}{2} |z|^2.
              \end{align*} In other words
              \begin{align*}
                u^\epsilon(t_0, y_0+z)-u^\epsilon(t_0, y_0) -q_0.z \ge -\frac{\vartheta}{2} |z|^2,
                \end{align*} which is exactly what we had claimed.

     \end{proof}

      \section*{Acknowledgement}
        We thank the anonymous referee for his/her comments and suggestions that have helped to improve this article.

\end{document}